\documentclass[twoside,reqno]{amsart}
\usepackage{amssymb,amsmath,amsfonts,amsthm}
\usepackage{enumerate}
\usepackage{amsfonts}
\usepackage{array}
 \usepackage[colorlinks=true]{hyperref}
\hypersetup{urlcolor=blue, citecolor=red}

\newtheorem{theorem}{Theorem}
\newtheorem{remark}[theorem]{Remark}

\newtheorem{proposition}[theorem]{Proposition}
\newtheorem{lemma}[theorem]{Lemma}

\newcommand{\R}{{\mathbb R}}

\def\duno{\partial_1}
\def\ddue{\partial_2}
\def\dtre{\partial_3}

\def\dt{\partial_t}

\def\ds{\displaystyle}
\def\eps{\epsilon}

\title[Plasma-vacuum interface]{Well-posedness of the linearized\\ plasma-vacuum interface problem}
\author[p. secchi and y. trakhinin]{}

\subjclass[2000]{Primary: 76W05; Secondary: 35Q35, 35L50, 76E17, 76E25, 35R35, 76B03}
\keywords{Ideal compressible Magneto-hydrodynamics, plasma-vacuum interface, free boundary}

\email{paolo.secchi@ing.unibs.it}
\email{trakhin@math.nsc.ru}

\thanks{The first author PS is supported by the national research project PRIN 2009 \lq\lq Equations of Fluid Dynamics of Hyperbolic Type and Conservation Laws\rq\rq. Part of this work was done during the fellowship of the second author YT at the Landau Network-Centro Volta-Cariplo Foundation spent at the Department of Mathematics of the University of Brescia in Italy.
YT would like to warmly thank the Department of Mathematics of the University of Brescia
for its kind hospitality during the visiting period.}

\date{\today}

\begin{document}
\maketitle

\centerline{\scshape Paolo Secchi}
\smallskip
{\footnotesize
%% please put the address of the second author
\centerline{Dipartimento di Matematica, Facolt\`a di Ingegneria, Universit\`a di Brescia}
\centerline{Via Valotti, 9, 25133 Brescia, Italy}
}

\medskip

\centerline{\scshape Yuri Trakhinin}
\smallskip
{\footnotesize
%% please put the address of the first author
\centerline{Sobolev Institute of Mathematics, Koptyug av. 4, 630090 Novosibirsk, Russia}
}
\bigskip

%\centerline{\bf PRELIMINARY VERSION}

\bigskip

\begin{abstract}
We consider the free boundary problem for the plasma-vacuum interface in ideal compressible magnetohydrodynamics (MHD). In the plasma region the flow is governed by the usual compressible MHD equations, while in the vacuum region we consider the {\it pre-Maxwell dynamics} for the magnetic field. At the free-interface we assume that the total pressure is continuous and that the magnetic field is tangent to the boundary. The plasma density does not go to zero continuously at the interface, but has a jump, meaning that it is bounded away from zero in the plasma region and it is identically zero in the vacuum region. Under a suitable stability
condition satisfied at each point of the plasma-vacuum interface,
we prove the well-posedness of the linearized problem in conormal Sobolev spaces.
\end{abstract}

\section{Introduction}
\label{sect1}

Consider the equations of ideal compressible MHD:
\begin{equation}
\left\{
\begin{array}{l}
\partial_t\rho  +{\rm div}\, (\rho {v} )=0,\\[3pt]
\partial_t(\rho {v} ) +{\rm div}\,(\rho{v}\otimes{v} -{H}\otimes{H} ) +
{\nabla}q=0, \\[3pt]
\partial_t{H} -{\nabla}\times ({v} {\times}{H})=0,\\[3pt]
\partial_t\bigl( \rho e +\frac{1}{2}|{H}|^2\bigr)+
{\rm div}\, \bigl((\rho e +p){v} +{H}{\times}({v}{\times}{H})\bigr)=0,
\end{array}
\right.
\label{1}
\end{equation}
where $\rho$ denotes density, $v\in\mathbb{R}^3$ plasma velocity, $H \in\mathbb{R}^3$ magnetic field, $p=p(\rho,S )$ pressure, $q =p+\frac{1}{2}|{H} |^2$ total pressure, $S$ entropy, $e=E+\frac{1}{2}|{v}|^2$ total energy, and  $E=E(\rho,S )$ internal energy. With a state equation of gas, $\rho=\rho(p ,S)$, and the first principle of thermodynamics, \eqref{1} is a closed system.

System (\ref{1}) is supplemented by the divergence constraint
\begin{equation}
{\rm div}\, {H} =0
\label{2}
\end{equation}
on the initial data. As is known, taking into account \eqref{2}, we can easily symmetrize system \eqref{1} by rewriting it in the nonconservative form
\begin{equation}
\left\{
\begin{array}{l}
{\displaystyle\frac{\rho_p}{\rho}}\,{\displaystyle\frac{{\rm d} p}{{\rm d}t} +{\rm div}\,{v} =0},\qquad
\rho\, {\displaystyle\frac{{\rm d}v}{{\rm d}t}-({H}\cdot\nabla ){H}+{\nabla}
q  =0 },\\[9pt]
{\displaystyle\frac{{\rm d}{H}}{{\rm d}t} - ({H} \cdot\nabla ){v} +
{H}\,{\rm div}\,{v}=0},\qquad
{\displaystyle\frac{{\rm d} S}{{\rm d} t} =0},
\end{array}\right. \label{3}
\end{equation}
where $\rho_p\equiv\partial\rho/\partial p$ and ${\rm d} /{\rm d} t =\partial_t+({v} \cdot{\nabla} )$.
A different symmetrization is obtained if we consider $q$ instead of $p$.
%%%
In terms of $q $ the equation for the pressure in \eqref{3} takes
the form
\begin{equation}
\begin{array}{ll}\label{equq}
\displaystyle\frac{\rho_p}{\rho }\left\{\frac{{\rm d} q}{{\rm d}t} -H
\cdot\displaystyle\frac{{\rm d} H}{{\rm d}t} \right\}+{\rm div}\,{v}=0,
\end{array}
\end{equation}
where it is understood that now
$\rho =\rho(q  -|H  |^2/2,S)$ and similarly for
$\rho_p $.
Then we derive ${\rm div}\,{v} $ from \eqref{equq} and rewrite the equation
for the magnetic field in \eqref{3} as
\begin{equation}
\begin{array}{ll}\label{equH}
\displaystyle\frac{{\rm d} H}{{\rm d}t} -(H \cdot\nabla)v  -
 \frac{\rho_p}{\rho }H\left\{\frac{{\rm d} q}{{\rm d}t} -H
\cdot\frac{{\rm d} H}{{\rm d}t} \right\}=0.
\end{array}
\end{equation}
Substituting \eqref{equq}, \eqref{equH} in \eqref{3} then gives the following symmetric system
\begin{equation}
\begin{array}{ll}\label{3'}
\left(\begin{matrix}
{\rho_p/\rho}&\underline
0&-({\rho_p/\rho})H &0 \\
\underline 0^T&\rho
I_3&0_3&\underline 0^T\\
-({\rho_p/\rho})H^T&0_3&I_3+({\rho_p/\rho})H\otimes H&\underline 0^T\\
0&\underline 0&\underline 0&1
\end{matrix}\right)\dt
\left(\begin{matrix}
q \\ v \\
H\\S \end{matrix}\right)+\\
 \\
 +
\left( \begin{matrix}
(\rho_p/\rho)
v \cdot\nabla&\nabla\cdot&-({\rho_p/\rho})Hv \cdot\nabla&0\\
\nabla&\rho v \cdot\nabla I_3&-H \cdot\nabla I_3&\underline 0^T\\
-({\rho_p/\rho})H^T v \cdot\nabla&-H \cdot\nabla I_3&(I_3+({\rho_p/\rho})H\otimes H)
v \cdot\nabla&\underline 0^T\\
0&\underline 0&\underline 0&v \cdot\nabla
\end{matrix}\right)
\left(\begin{matrix}q \\ v \\ H\\S \end{matrix}\right)
=0\,,
\end{array}
\end{equation}
where $\underline 0=(0,0,0)$.
%%%
Given this symmetrization, as the unknown we can choose the vector $ U =U (t, x )=(q, v,H, S)$. For the sake of brevity we write system (\ref{3'}) in the form
\begin{equation}
\label{4}
A_0(U )\partial_tU+\sum_{j=1}^3A_j(U )\partial_jU=0,
\end{equation}
which is symmetric hyperbolic provided the hyperbolicity condition $A_0>0$ holds:
\begin{equation}
\rho  >0,\quad \rho_p >0. \label{5}
\end{equation}

Plasma-vacuum interface problems for system \eqref{1} appear in the mathematical modeling of plasma confinement by magnetic fields (see, e.g., \cite{Goed}). In this model the plasma is confined inside a perfectly conducting rigid wall and isolated from it by  a vacuum region, due to the effect of strong magnetic fields.
This subject is very popular since the 1950--70's, but most of theoretical studies are devoted to finding stability criteria of equilibrium states. The typical work in this direction is the classical paper of Bernstein et al.
\cite{BFKK}.  In astrophysics, the plasma-vacuum interface problem can be used for modeling the motion of a star or the solar corona when magnetic fields are taken into account.

According to our knowledge there are still no well-posedness results for full ({\it non-stationary}) plasma-vacuum models. More precisely, an energy a priori estimate in Sobolev spaces for the linearization of a plasma-vacuum interface problem (see its description just below) was proved in \cite{trakhinin10}, but the existence of solutions to this problem remained open. In fact, the proof of existence of solutions is the main goal of the present paper.

Let $\Omega^+(t)$ and $\Omega ^-(t)$ be space-time domains occupied by the plasma and the vacuum respectively. That is, in the domain
$\Omega^+(t)$ we consider system \eqref{1} (or \eqref{4}) governing the motion of an ideal plasma and in the domain $\Omega^-(t)$  we have the elliptic (div-curl) system
\begin{equation}
\nabla \times \mathcal{H} =0,\qquad {\rm div}\, \mathcal{H}=0,\label{6}
\end{equation}
describing the vacuum magnetic field $\mathcal{H}\in\mathbb{R}^3$. Here, as in \cite{BFKK,Goed},
we consider so-called {\it pre-Maxwell dynamics}. That is, as usual in nonrelativistic MHD, we neglect the displacement current $(1/c)\,\partial_tE$, where $c$ is the speed of light and $E$ is the electric field.

Let us assume that the interface between plasma and vacuum is given by a hypersurface $\Gamma (t)=\{F(t,x)=0\}$. It is to be determined and moves with the velocity of plasma particles at the boundary:
\begin{equation}
\frac{{\rm d}F }{{\rm d} t}=0\quad \mbox{on}\ \Gamma (t)\label{7}
\end{equation}
(for all $t\in [0,T]$). As $F$ is an unknown of the problem, this is a free-boundary problem. The plasma variable $U$ is connected with the vacuum magnetic field
$\mathcal{H}$  through the relations
\cite{BFKK,Goed}
\begin{equation}
 [q]=0,\quad  H\cdot N=0, \quad   \mathcal{H}\cdot N=0,\quad \mbox{on}\ \Gamma (t),
\label{8}
\end{equation}
where $N=\nabla F$ and $[q]= q|_{\Gamma}-\frac{1}{2}|\mathcal{H}|^2_{|\Gamma}$ denotes the jump of the total pressure across the interface.
These relations together with \eqref{7} are the boundary conditions at the interface $\Gamma (t)$.

As in \cite{lindblad2,trakhinin09}, we will assume that for problem \eqref{1}, \eqref{6}--\eqref{8} the hyperbolicity conditions \eqref{5} are assumed to be satisfied in $\Omega^+(t)$ up to the boundary $\Gamma(t)$, i.e., the plasma density does not go to zero continuously, but has a jump (clearly in the vacuum region $\Omega^-(t)$ the density is identically zero). This assumption is compatible with the continuity of the total pressure in \eqref{8}.

Since the interface moves with the velocity of plasma particles at the boundary, by introducing the Lagrangian coordinates one can reduce the original problem to that in a fixed domain. This approach has been recently employed with success in a series of papers on the Euler equations in vacuum, see \cite{coutandlindbladshkoller,coutandshkoller1,coutandshkoller2,coutandshkoller3,lindblad2}. However, as, for example, for contact discontinuities in various models of fluid dynamics (e.g., for current-vortex sheets \cite{cmst,trakhinin09arma}), this approach seems hardly applicable for problem \eqref{1}, \eqref{6}--\eqref{8}. Therefore, we will work in the Eulerian coordinates and for technical simplicity we will assume that the space-time domains $\Omega^\pm (t)$ have the following form.

%%%%%%%%%%%%%%%%%%%%%%%%%%%%%%%%%%%%%%%%%%

Let us assume that the moving interface $\Gamma(t)$ takes the form $$
\Gamma(t) := \{ (x_1,x') \in \R^3 \, , \, x_1=\varphi(t,x')\} \, ,
$$
where $t \in [0,T]$ and $x'=(x_2,x_3)$. Then we have $\Omega^\pm(t)=\{x_1\gtrless \varphi(t,x')\}$.
With our parametrization of $\Gamma (t)$, an equivalent formulation
of the boundary conditions \eqref{7}, \eqref{8} at the interface is
\begin{equation}
\partial_t\varphi =v_N,\quad [q]=0,\quad H_N=0,\quad \mathcal{H}_N=0 \quad \mbox{on}\ \Gamma (t),\label{15}
\end{equation}
where $v_N=v\cdot N$, $H_N=H\cdot N$, $\mathcal{H}_N=\mathcal{H}\cdot N$, $N=(1,-\partial_2\varphi ,-\partial_3\varphi )$.
\\
System \eqref{4}, \eqref{6}, \eqref{15} is supplemented with initial conditions
\begin{equation}
\begin{array}{ll}
{U} (0,{x})={U}_0({x}),\quad {x}\in \Omega^{+} (0),\qquad
\varphi(0,{x})=\varphi_0({x}),\quad {x}\in\Gamma ,\\
\mathcal{H}(0,x)=
\mathcal{H}_0(x),\quad {x}\in \Omega^{-}(0),\label{11}
\end{array}
\end{equation}
From the mathematical point of view, a natural wish is to find conditions on the initial data
providing the existence and uniqueness on some time interval $[0,T]$ of a solution $({U},\mathcal{H},\varphi)$ to problem \eqref{4}, \eqref{6}, \eqref{15}, \eqref{11} in Sobolev spaces. Since \eqref{1} is a system of hyperbolic conservation laws that can produce shock waves and other types of strong discontinuities
(e.g., current-vortex sheets \cite{trakhinin09arma}), it is natural to expect to obtain only local-in-time existence theorems.

We must regard the boundary conditions on $H$ in \eqref{15} as the restriction on the initial data \eqref{11}. More precisely, we can prove that a solution of  \eqref{4}, \eqref{15} (if it exists for all $t\in [0,T]$) satisfies
\[
{\rm div}\, {H} =0 \quad \mbox{in}\ \Omega^+ (t)\quad \mbox{and}\quad  H_N=0\quad \mbox{on}\ \Gamma (t),
\]
for all $t\in [0,T]$, if the latter were satisfied at $t=0$, i.e., for the initial data \eqref{11}. In particular, the fulfillment of ${\rm div}\, {H} =0$ implies that systems \eqref{1} and \eqref{4} are equivalent on solutions of problem \eqref{4}, \eqref{15}, \eqref{11}.

%%%%%%%%%%%%%%%%%%%%%%%%%%%%%%%%%%%%%%%%%%%%%%%%%%%%%%%%%%%%%%%%%%%
\subsection{An equivalent formulation in the fixed domain}

Let us denote
$$
\Omega^\pm := \R^3 \cap \{x_1\gtrless 0 \} \, ,\qquad \Gamma :=\R^3\cap\{x_1=0\} \, .
$$
We want to reduce the free boundary problem \eqref{4}, \eqref{6},  \eqref{15}, \eqref{11} to the fixed
domains $\Omega^\pm$. For this purpose we introduce a suitable change of variables that is inspired by Lannes \cite{lannes}. In all what follows, $H^s(\omega)$ denotes the Sobolev space of order $s$ on a domain $\omega$.
The following lemma shows how to lift functions from
$\Gamma$ to $\R^3$. An important point is the regularization of one half derivative of the lifting function $\Psi$ w.r.t. the given function $\varphi$. For instance, there is no such regularization in the lifting function chosen in \cite{majda1,metivier2}.

\begin{lemma}
\label{lemma1}

Let $m\ge 3$. For all $\eps>0$ there exists a continuous linear map $\varphi\in H^{m-0.5}(\R^2) \mapsto
\Psi \in H^m(\R^3)$ such that $\Psi(0,x')=\varphi(x')$, $\duno \Psi (0,x')=0$ on $\Gamma$, and
\begin{equation}
\begin{array}{ll}\label{diffeopiccolo}
\|\duno\Psi\|_{L^\infty(\R^3)}\le\eps \, \| \varphi\|_{H^{2}(\R^2)}.
\end{array}
\end{equation}

\end{lemma}
%%%

\noindent We give the proof of Lemma \ref{lemma1} in Section \ref{prooflemma1}
at the end of this article. The following lemma gives the time-dependent version of Lemma \ref{lemma1}.

\begin{lemma}
\label{lemma2}
Let $m \ge3$ be an integer and let $T>0$. For all $\eps>0$ there exists a continuous linear map $\varphi\in \cap_{j=0}^{m-1}
{\mathcal C}^j([0,T];H^{m-j-0.5}(\R^2)) \mapsto \Psi \in \cap_{j=0}^{m-1} {\mathcal C}^j([0,T];H^{m-j}(\R^3))$
such that $\Psi(t,0,x')=\varphi(t,x')$,  $\duno \Psi (t,0,x')=0$ on $\Gamma$, and
\begin{equation}
\begin{array}{ll}\label{diffeopiccolo2}
\|\duno\Psi\|_{{\mathcal C}([0,T];L^\infty(\R^3))}\le\eps \, \| \varphi\|_{{\mathcal C}([0,T];H^{2}(\R^2))}.
\end{array}
\end{equation}
Furthermore, there exists a constant $C>0$ that is independent of $T$ and only depends on $m$, such that
\begin{multline*}
\forall \, \varphi\in \cap_{j=0}^{m-1} {\mathcal C}^j([0,T];H^{m-j-0.5}(\R^2)) \, ,\quad
\forall \, j=0,\dots,m-1 \, ,\quad \forall \, t \in [0,T] \, ,\\
\| \partial_t^j \Psi (t,\cdot) \|_{H^{m-j}(\R^3)} \le C \,
\| \partial_t^j \varphi (t,\cdot) \|_{H^{m-j-0.5}(\R^2)} \, .
\end{multline*}
\end{lemma}
%%%

\noindent The proof of Lemma \ref{lemma2} is also postponed to Section \ref{prooflemma1}. The diffeomorphism
that reduces the free boundary problem \eqref{4}, \eqref{15}, \eqref{11} to the
fixed domains $\Omega^\pm$ is given in the following lemma.

\begin{lemma}
\label{lemma3}
Let $m \ge3$ be an integer. For all $T>0$, and
for all $\varphi \in \cap_{j=0}^{m-1} {\mathcal C}^j([0,T];H^{m-j-0.5}(\R^2))$, satisfying without loss of generality $\| \varphi\|_{{\mathcal C}
([0,T];H^{2}(\R^2))} \le 1$, there exists a function $\Psi \in \cap_{j=0}^{m-1} {\mathcal C}^j([0,T];H^{m-j}(\R^3))$
such that
the function
\begin{equation}
\label{change}
\Phi(t,x) := \big( x_1 +\Psi(t,x),x' \big) \, , \qquad (t,x) \in [0,T]\times \R^3 \, ,
\end{equation}
defines an $H^m$-diffeomorphism of $\R^3$ for all $t \in [0,T]$.
Moreover, there holds $\partial^j_t (\Phi - Id) \in {\mathcal C}([0,T];H^{m-j}(\R^3))$ for $j=0,\dots, m-1$,
$\Phi(t,0,x')=(\varphi(t,x'),x')$, $\duno \Phi(t,0,x')=(1,0,0)$.
\end{lemma}

\begin{proof}[Proof of Lemma \ref{lemma3}]
The proof follows directly from Lemma \ref{lemma2} because
$$
\duno \Phi_1(t,x) =1 +\duno \Psi(t,x)\ge 1-\| \duno\Psi(t,\cdot) \|_{{\mathcal C}([0,T];L^{\infty}(\R^3))}
\ge 1 -\eps \, \|\varphi\|_{{\mathcal C}([0,T];H^{2}(\R^2))} \ge 1/2 \, ,
$$
provided $\eps$ is taken sufficiently small, e.g. $\eps<1/2$. The other properties of $\Phi$ follow directly from Lemma \ref{lemma2}.
\end{proof}

We introduce the change of independent variables defined by \eqref{change} by setting
\[
\widetilde{U}(t,x ):= {U}(t,\Phi (t,x)),\quad
\widetilde{\mathcal{H}}(t,x ):= \mathcal{H}(t,\Phi (t,x)).
\]
Dropping for convenience tildes in $\widetilde{U}$ and $\widetilde{\mathcal{H}}$, problem \eqref{4}, \eqref{6} \eqref{15}, \eqref{11} can be reformulated on the fixed
reference domains $\Omega^\pm$ as
\begin{equation}
\mathbb{P}(U,\Psi)=0\quad\mbox{in}\ [0,T]\times \Omega^+,\quad \mathbb{V}(\mathcal{H},\Psi)=0\quad\mbox{in}\ [0,T]\times \Omega^-,\label{16}
\end{equation}
\begin{equation}
\mathbb{B}(U,\mathcal{H},\varphi )=0\quad\mbox{on}\ [0,T] \times\Gamma,\label{17}
\end{equation}
\begin{equation}
(U,\mathcal{H})|_{t=0}=(U_0,\mathcal{H}_0)\quad\mbox{in}\ \Omega^+\times\Omega^-,\qquad \varphi|_{t=0}=\varphi_0\quad \mbox{on}\ \Gamma,\label{18}
\end{equation}
where $\mathbb{P}(U,\Psi)=P(U,\Psi)U$,
\[
P(U,\Psi)=A_0(U)\partial_t +\widetilde{A}_1(U,\Psi)\partial_1+A_2(U )\partial_2+A_3(U )\partial_3,
\]
\[
\widetilde{A}_1(U,\Psi )=\frac{1}{\partial_1\Phi_1}\Bigl(
A_1(U )-A_0(U)\partial_t\Psi -\sum_{k=2}^3A_k(U)\partial_k\Psi \Bigr),
\]
\[
\mathbb{V}(\mathcal{H},\Psi)=\left(
\begin{array}{c}
\nabla\times \mathfrak{H}\\
{\rm div}\,\mathfrak{h}
\end{array}
\right),
\]
\[
\mathfrak{H}=(\mathcal{H}_1\partial_1\Phi,\mathcal{H}_{\tau_2},\mathcal{H}_{\tau_3}),\quad
\mathfrak{h}=(\mathcal{H}_{N},\mathcal{H}_2\partial_1\Phi_1,\mathcal{H}_3\partial_1\Phi_1),
\]
\[
\mathcal{H}_{N}=\mathcal{H}_1-\mathcal{H}_2\partial_2\Psi -\mathcal{H}_3\partial_3\Psi,\quad
\mathcal{H}_{\tau_i}=\mathcal{H}_1\partial_i\Psi +\mathcal{H}_i,\quad i=2,3,
\]
\[
\mathbb{B}(U,\mathcal{H},\varphi )=\left(
\begin{array}{c}
\partial_t\varphi -v_{N |x_1=0}\\ {[}q{]} \\ \mathcal{H}_{N |x_1=0}
\end{array}
\right),\quad [q]=q_{|x_1=0}-\frac{1}{2}|\mathcal{H}|^2_{x_1=0}
 ,
\]
\[
v_{N}=v_1- v_2\partial_2\Psi - v_3\partial_3\Psi .
\]

To avoid an overload of notation we have denoted by the same symbols $v_N,\mathcal{H}_N$ here above and $v_N,\mathcal{H}_N$ as in \eqref{15}. Notice that
$v_{N |x_1=0}=v_1- v_2\partial_2\varphi - v_3\partial_3\varphi,$ $\mathcal{H}_{N |x_1=0}=
\mathcal{H}_1- \mathcal{H}_2\partial_2\varphi - \mathcal{H}_3\partial_3\varphi $, as in the previous definition in \eqref{15}.

We did not include in problem \eqref{16}--\eqref{18} the equation
\begin{equation}
{\rm div}\, h=0\quad\mbox{in}\ [0,T]\times \Omega^+,\label{19}
\end{equation}
and the boundary condition
\begin{equation}
H_{N}=0\quad\mbox{on}\ [0,T]\times\Gamma,\label{20}
\end{equation}
where $h=(H_{N},H_2\partial_1\Phi_1,H_3\partial_1\Phi_1)$,
$H_{N}=H_1-H_2\partial_2\Psi -H_3\partial_3\Psi$,
because they are just restrictions on the initial data \eqref{18}. More precisely, referring to \cite{trakhinin09arma} for the proof, we have the following proposition.

\begin{proposition}
Let the initial data \eqref{18} satisfy \eqref{19} and \eqref{20} for $t=0$.
If $(U,\mathcal{H},\varphi )$ is a solution  of problem \eqref{16}--\eqref{18}, then this solution satisfies \eqref{19} and \eqref{20} for all $t\in [0,T]$.
\label{p1}
\end{proposition}

Note that Proposition \ref{p1} stays valid if in \eqref{16} we replace system $\mathbb{P}(U,\Psi)=0$
by system \eqref{1} in the straightened variables. This means that these systems are equivalent on solutions of our plasma-vacuum interface problem and we may justifiably replace
the conservation laws \eqref{1} by their nonconservative form \eqref{4}.

%%%%%%%%%%%%%%%%%%%%%%%%%%%%%%%%%%%%%%%%
%%%%%%%%%%%%%%%%%%%%%%%%%%%%%%%%%%%%%%%%%%%%%%
%%%%%%%%%%%%%%%%%%%%%%%%%%%%%%%%%%%%%%%%%%%%%%
%%%%%%%%%%%%%%%%%%%%%%%%%%%%%%%%%%%%%%%%%%%%%%
%%%%%%%%%%%%%%%%%%%%%%%%%%%%%%%%%%%%%%%%%%%%%%
%%%%%%%%%%%%%%%%%%%%%%%%%%%%%%%%%%%%%%%%%%%%%%
\section{The linearized problem}
%%%%%%%%%%%%%%%%%%%%%%%%%%%%%%%%%%%%%%%%%%%%%%
\subsection{Basic state}
\label{s2.2}

Let us denote
\[
Q^\pm_T:= (-\infty,T]\times\Omega^\pm,\quad \omega_T:=(-\infty,T]\times\Gamma.
\]
Let
\begin{equation}
(\widehat{U}(t,x ),\widehat{\mathcal{H}}(t,x ),\hat{\varphi}(t,{x}'))
\label{21}
\end{equation}
be a given sufficiently smooth vector-function
with $\widehat{U}=(\hat{q},\hat{v},\widehat{H},\widehat{S})$, respectively defined on $Q^+_T,Q^-_T,\omega_T$, with
\begin{equation}
\begin{array}{ll}\label{}
\|\widehat{U}\|_{W^{2,\infty}(Q^+_T)}+
\|\partial_1\widehat{U}\|_{W^{2,\infty}(Q^+_T)}+
\|\widehat{\mathcal{H}}\|_{W^{2,\infty}(Q^-_T)}+
\|\hat{\varphi}\|_{W^{3,\infty}([0,T]\times\R^2)} \leq K,\\
\\
 \| \hat\varphi\|_{{\mathcal C}
([0,T];H^{2}(\R^2))} \le 1,
\label{22}
\end{array}
\end{equation}
where $K>0$ is a constant.
Corresponding to the given $\hat\varphi$ we construct $\hat\Psi$ and the diffeomorphism $\hat\Phi$  as in Lemmata \ref{lemma2} and \ref{lemma3} such that
\[
\partial_1\widehat{\Phi}_1\geq 1/2.
\]
We assume that the basic state \eqref{21} satisfies (for some positive $\rho_0,\rho_1\in\R$)
\begin{equation}
\rho (\hat{p},\widehat{S})\geq \rho_0 >0,\quad \rho_p(\hat{p},\widehat{S})\ge \rho_1 >0 \qquad \mbox{in}\ \overline{Q}^+_T,
\label{23}
\end{equation}
\begin{equation}
\partial_t\widehat{H}+\frac{1}{\partial_1\widehat{\Phi}_1}\left\{ (\hat{w} \cdot\nabla )
\widehat{H} - (\hat{h} \cdot\nabla ) \hat{v} + \widehat{H}{\rm div}\,\hat{u}\right\} =0\qquad \mbox{in}\ Q^+_T,
\label{26}
\end{equation}
\begin{equation}
\nabla\times \widehat{\mathfrak{H}}=0,\quad {\rm div}\,\hat{\mathfrak{h}}=0\qquad \mbox{in}\ Q^-_T,
\label{25}
\end{equation}
\begin{equation}
\partial_t\hat{\varphi}-\hat{v}_{N}=0,\quad \widehat{\mathcal{H}}_N=0 \qquad \mbox{on}\,\; \omega_T,\label{24}
\end{equation}
where all the ``hat'' values are determined like corresponding values for $(U,\mathcal{H},\varphi)$, i.e.
\[
\widehat{\mathfrak{H}}=(\widehat{\mathcal{H}}_1\partial_1\widehat{\Phi}_1,
\widehat{\mathcal{H}}_{\tau_2},\widehat{\mathcal{H}}_{\tau_3}),\quad
\hat{\mathfrak{h}}=(\hat{\mathcal{H}}_{N},\hat{\mathcal{H}}_2\partial_1\widehat{\Phi}_1,
\hat{\mathcal{H}}_3\partial_1\widehat{\Phi}_1),
\quad \hat{h}=(\hat{H}_{N},\hat{H}_2\partial_1\hat{\Phi}_1,\hat{H}_3\partial_1\hat{\Phi}_1),
\]
%%%
\[
\hat p=\hat q  -|\hat H  |^2/2 ,\quad
\hat{v}_{N}=\hat{v}_1- \hat{v}_2\partial_2\hat{\Psi}- \hat{v}_3\partial_3\hat{\Psi},\quad
\hat{\mathcal{H}}_{N}=\hat{\mathcal{H}}_1- \hat{\mathcal{H}}_2\partial_2\hat{\Psi}- \hat{\mathcal{H}}_3\partial_3\hat{\Psi},
\]
and where
\[
\hat{u}=(\hat{v}_{N},\hat{v}_2\partial_1\widehat{\Phi}_1,\hat{v}_3\partial_1\widehat{\Phi}_1),\quad
\hat{w}=\hat{u}-(\partial_t\widehat{\Psi},0,0).
\]
Note that (\ref{22}) yields
\[
 \|\nabla_{t,x}\widehat{\Psi}\|_{W^{2,\infty}([0,T]\times\R^3)}\leq C(K),
\]
where $\nabla_{t,x}=(\partial_t, \nabla )$ and $C=C(K)>0$ is a constant depending on $K$.
\\
It follows from (\ref{26}) that the constraints
\begin{equation}
{\rm div}\,\hat{h}=0\quad \mbox{in}\; Q^+_T,\quad \widehat{H}_{N}=0\quad \mbox{on}\,\; \omega_T,
\label{27}
\end{equation}
are satisfied for the basic state (\ref{21}) if they hold at $t=0$ (see \cite{trakhinin09arma} for the proof).
Thus, for the basic state we also require the fulfillment of conditions
\eqref{27} at $t=0$.

%%%%%%%%%%%%%%%%%%%%%%%%%%%%%%%%
%%%%%%%%%%%%%%%%%%%%%%%%%%%%%%%%
%%%%%%%%%%%%%%%%%%%%%%%%%%%%%%%%
%%%%%%%%%%%%%%%%%%%%%%%%%%%%%%%%

\subsection{Linearized problem}
\label{s2.3}

The linearized equations for (\ref{16}), (\ref{17}) read:
\[
\mathbb{P}'(\widehat{U},\widehat{\Psi})(\delta U,\delta\Psi):=
\frac{\rm d}{{\rm d}\varepsilon}\mathbb{P}(U_{\varepsilon},\Psi_{\varepsilon})|_{\varepsilon =0}=f
\qquad \mbox{in}\ Q^+_T,
\]
\[
\mathbb{V}'(\widehat{\mathcal{H}},\widehat{\Psi})(\delta \mathcal{H},\delta\Psi):=
\frac{\rm d}{{\rm d}\varepsilon}\mathbb{V}(\mathcal{H}_{\varepsilon},\Psi_{\varepsilon})|_{\varepsilon =0}=\mathcal{G}'
\qquad \mbox{in}\ Q^-_T,
\]
\[
\mathbb{B}'(\widehat{U},\widehat{\mathcal{H}},\hat{\varphi})(\delta U,\delta \mathcal{H},\delta \varphi ):=
\frac{\rm d}{{\rm d}\varepsilon}\mathbb{B}(U_{\varepsilon},\mathcal{H}_{\varepsilon},\varphi_{\varepsilon})|_{\varepsilon =0}={g}
\qquad \mbox{on}\ \omega_T,
\]
where $U_{\varepsilon}=\widehat{U}+ \varepsilon\,\delta U$, $\mathcal{H}_{\varepsilon}=
\widehat{\mathcal{H}}+\varepsilon\,\delta \mathcal{H}$,
$\varphi_{\varepsilon}=\hat{\varphi}+ \varepsilon\,\delta \varphi$;
$\delta\Psi$ is constructed from $\delta \varphi$ as in Lemma \ref{lemma2} and
$\Psi_{\varepsilon}=\hat\Psi +  \varepsilon\,\delta\Psi$.

Here we introduce the source terms $f=(f_1,\ldots ,f_8)$, $\mathcal{G}'=(\chi,
\Xi)$, $\chi=(\chi_1, \chi_2,
\chi_3)$, and $g=(g_1,g_2,g_3)$ to make the interior equations and the boundary conditions inhomogeneous.

We compute the exact form of the linearized equations (below we drop $\delta$):
\[
\mathbb{P}'(\widehat{U},\widehat{\Psi})(U,\Psi)
=
P(\widehat{U},\widehat{\Psi})U +{\mathcal C}(\widehat{U},\widehat{\Psi})
U -   \bigl\{L(\widehat{U},\widehat{\Psi})\Psi\bigr\}\frac{\partial_1\widehat{U}}{\partial_1\widehat{\Phi}_1}
=f,
\]
\[
\mathbb{V}'(\widehat{\mathcal{H}},\widehat{\Psi})(\mathcal{H},\Psi)=
\mathbb{V}(\mathcal{H},\widehat{\Psi})+
\left(\begin{array}{c}
\nabla\widehat{\mathcal{H}}_1\times\nabla\Psi\\[3pt]
\nabla \times \left(\begin{array}{c} 0 \\ -\widehat{\mathcal{H}}_3 \\
\widehat{\mathcal{H}}_2 \end{array} \right) \cdot \nabla\Psi
\end{array}
\right)=\mathcal{G}',
\]
\[
\mathbb{B}'(\widehat{U},\widehat{\mathcal{H}},\hat{\varphi})(U,\mathcal{H},\varphi )=
\left(
\begin{array}{c}
\partial_t\varphi +\hat{v}_2\partial_2\varphi+\hat{v}_3\partial_3\varphi -v_{N}\\[3pt]
q-\widehat{\mathcal{H}} \cdot \mathcal{H}\\[3pt]
\mathcal{H}_N-\widehat{\mathcal{H}}_2\partial_2\varphi -\widehat{\mathcal{H}}_3\partial_3\varphi
\end{array}
\right)_{|x_1=0}=g,
\]
where $q:=p+ \widehat{H}\cdot H$, $v_{N}:= v_1-v_2\partial_2\widehat{\Psi}-v_3\partial_3\widehat{\Psi}$, and the matrix
${\mathcal C}(\widehat{U},\widehat{\Psi})$ is determined as follows:
\[
\begin{array}{r}
{\mathcal C}(\widehat{U},\widehat{\Psi})Y
= (Y ,\nabla_yA_0(\widehat{U} ))\partial_t\widehat{U}
 +(Y ,\nabla_y\widetilde{A}_1(\widehat{U},\widehat{\Psi}))\partial_1\widehat{U}
 \\[6pt]
+ (Y ,\nabla_yA_2(\widehat{U} ))\partial_2\widehat{U}
+ (Y ,\nabla_yA_3(\widehat{U} ))\partial_3\widehat{U},
\end{array}
\]
\[
(Y ,\nabla_y A(\widehat{U})):=\sum_{i=1}^8y_i\left.\left(\frac{\partial A (Y )}{
\partial y_i}\right|_{Y =\widehat{U}}\right),\quad Y =(y_1,\ldots ,y_8).
\]
Since the differential operators $\mathbb{P}'(\widehat{U},\widehat{\Psi})$ and $\mathbb{V}'(\widehat{\mathcal{H}},\widehat{\Psi})$ are first-order operators in $\Psi$,
as in  \cite{alinhac} the linearized problem is rewritten in terms of the ``good unknown''
\begin{equation}
\dot{U}:=U -\frac{\Psi}{\partial_1\widehat{\Phi}_1}\,\partial_1\widehat{U},\quad
\dot{\mathcal{H}}:=\mathcal{H} -\frac{\Psi}{\partial_1\widehat{\Phi}_1}\,\partial_1\widehat{\mathcal{H}}.
\label{29}
\end{equation}
Taking into account assumptions \eqref{24} and \eqref{25} and omitting detailed calculations,
we rewrite our linearized equations in terms of the new unknowns \eqref{29}:
\begin{equation}
P(\widehat{U},\widehat{\Psi})\dot{U} +{\mathcal C}(\widehat{U},\widehat{\Psi})
\dot{U} - \frac{\Psi}{\partial_1\widehat{\Phi}_1}\,\partial_1\bigl\{\mathbb{L}
(\widehat{U},\widehat{\Psi})\bigr\}=f,
\label{30}
\end{equation}
\begin{equation}
\mathbb{V}(\dot{\mathcal{H}},\widehat{\Psi})=\mathcal{G}'.
\label{31}
\end{equation}
\begin{multline}
\mathbb{B}'(\widehat{U},\widehat{\mathcal{H}},\hat{\varphi})(\dot{U},\dot{\mathcal{H}},\varphi ):= \mathbb{B}'(\widehat{U},\widehat{\mathcal{H}},\hat{\varphi})(U,\mathcal{H},\varphi )\\[6pt] =
 \left(
\begin{array}{c}
\partial_t\varphi+\hat{v}_2\partial_2\varphi+\hat{v}_3\partial_3\varphi-\dot{v}_{N}-
\varphi\,\partial_1\hat{v}_{N}\\[3pt]
\dot{q}-\widehat{\mathcal{H}} \cdot \dot{\mathcal{H}}+ [\partial_1\hat{q}]\varphi \\[3pt]
\dot{\mathcal{H}}_{N}-\partial_2\bigl(\widehat{\mathcal{H}}_2\varphi \bigr) -\partial_3\bigl(\widehat{\mathcal{H}}_3\varphi \bigr)
\end{array}\right)_{|x_1=0}=g,
\label{32}
\end{multline}
where $\dot{v}_{\rm N}=\dot{v}_1-\dot{v}_2\partial_2\hat{\Psi}-\dot{v}_3\partial_3\hat{\Psi}$,
$\dot{\mathcal{H}}_{N}=\dot{\mathcal{H}}_1-\dot{\mathcal{H}}_2\partial_2\hat{\Psi}-\dot{\mathcal{H}}_3\partial_3\hat{\Psi}$, and
\[
[\partial_1\hat{q}]=(\partial_1\hat{q})|_{x_1=0}-(\widehat{\mathcal{H}} \cdot \partial_1\widehat{\mathcal{H}})|_{x_1=0}.
\]
We used the last equation in \eqref{25} taken at $x_1=0$ while writing down the last boundary condition in \eqref{32}.

As in \cite{alinhac, CS, trakhinin09arma}, we drop the zeroth-order term in $\Psi$ in (\ref{30}) and consider the effective linear operators
\[
\mathbb{P}'_e(\widehat{U},\widehat{\Psi})\dot{U} :=P(\widehat{U},\widehat{\Psi})\dot{U} +{\mathcal C}(\widehat{U},\widehat{\Psi})
\dot{U}=f.
\]
In the future nonlinear analysis the dropped term in (\ref{30}) should be considered as an error term.
%Regarding \eqref{31} and \eqref{32}, without loss of generality we may actually drop the source term $\mathcal{F}$ and the boundary datum $g_3$. In fact, it follows from the detailed analysis of an exact form of the accumulated errors in the Nash-Moser iterative scheme for the elliptic system $\mathbb{V}(\mathcal{H},\Psi)=0$ and the boundary condition $\mathcal{H}_N|_{x_1=0}=0$ in \eqref{15}, that the source terms ${\mathcal F}$ and $g_3$ have the following special form:
%\begin{equation}
%\mathcal{F} =\left(
%\begin{array}{c}
%\nabla\times\mathfrak{B}\\ {\rm div}\,\mathfrak{b}
%\end{array}
%\right),\quad g_3=\mathfrak{b}_1|_{x_1=0},\label{33}
%\end{equation}
%where
%\[
%\mathfrak{B}=(b_1\partial_1\Phi_1,b_{\tau_2},b_{\tau_3}),\quad
%\mathfrak{b}=(b_{n},b_2\partial_1\Phi_1,b_3\partial_1\Phi_1),
%\]
%\[
%b_{n}=b_1-b_2\partial_2\Psi-b_3\partial_3\Psi,\quad
%b_{\tau_i}=b_1\partial_i\Psi+b_i,\quad i=2,3,
%\]
%and $b(t,x) =(b_1,b_2,b_3)$ is a vector-function. Passing to the new unknown
%\[
%\dot{\mathcal{H}}'=\dot{\mathcal{H}}-b
%\]
%and omitting then the primes, in view of \eqref{33}, we get the homogeneous system
%\[
%\mathbb{V}(\dot{\mathcal{H}},\widehat{\Psi})=0
%\]
%and the last boundary condition in \eqref{32} becomes homogeneous ($g_3=0$).
The new form of our linearized problem for $(\dot{U},\dot{\mathcal{H}},\varphi )$ reads:
%%%%%
\begin{subequations}\label{34}
\begin{align}
\widehat{A}_0\partial_t\dot{U}+\sum_{j=1}^{3}\widehat{A}_j\partial_j\dot{U}+
\widehat{\mathcal C}\dot{U}=f \qquad &\mbox{in}\ Q^+_T,\label{34a}
\\
\nabla\times \dot{\mathfrak{H}}=\chi,\quad {\rm div}\,\dot{\mathfrak{h}}=\Xi \qquad &\mbox{in}\ Q^-_T, \label{36}
\\
\partial_t\varphi=\dot{v}_{N}-\hat{v}_2\partial_2\varphi-\hat{v}_3\partial_3\varphi +
\varphi\,\partial_1\hat{v}_{N}+g_1,  \qquad &\label{34b}
\\
\dot{q}=\widehat{\mathcal{H}}\cdot\dot{\mathcal{H}}-  [ \partial_1\hat{q}] \varphi +g_2, \qquad & \label{35}
\\
\dot{\mathcal{H}}_{N} =\partial_2\bigl(\widehat{\mathcal{H}}_2\varphi \bigr) +\partial_3\bigl(\widehat{\mathcal{H}}_3\varphi \bigr)+g_3\qquad &\mbox{on}\ \omega_T,
\label{37}
\\
(\dot{U},\dot{\mathcal{H}},\varphi )=0\qquad &\mbox{for}\ t<0,\label{38a}
\end{align}
\end{subequations}
%%%%

where
\[
\widehat{A}_{\alpha}=:{A}_{\alpha}(\widehat{U}),\quad \alpha =0,2,3,\quad
\widehat{A}_1=:\widetilde{A}_1(\widehat{U},\widehat{\Psi}),\quad
\widehat{\mathcal C}:={\mathcal C}(\widehat{U},\widehat{\Psi}),
\]
\[
\dot{\mathfrak{H}}=(\dot{\mathcal{H}}_1\partial_1\widehat{\Phi}_1,\dot{\mathcal{H}}_{\tau_2},\dot{\mathcal{H}}_{\tau_3}),\quad
\dot{\mathfrak{h}}=(\dot{\mathcal{H}}_{N},\dot{\mathcal{H}}_2\partial_1\widehat{\Phi}_1,\dot{\mathcal{H}}_3\partial_1\widehat{\Phi}_1),
\]
\[
\dot{\mathcal{H}}_{N}=\dot{\mathcal{H}}_1-\dot{\mathcal{H}}_2\partial_2\widehat{\Psi}-\dot{\mathcal{H}}_3\partial_3\widehat{\Psi},\quad
\dot{\mathcal{H}}_{\tau_i}=\dot{\mathcal{H}}_1\partial_i\widehat{\Psi}+\dot{\mathcal{H}}_i,\quad i=2,3.
\]
The source term $ \chi$ of the first equation in \eqref{36} should satisfy the constraint ${\rm div}\, \chi=0$. For the resolution of the elliptic problem \eqref{36}, \eqref{37} the data $\Xi,g_3$ must satisfy the necessary compatibility condition
\begin{equation}
\begin{array}{ll}\label{compatibility}
\ds \int_{\Omega^-}\Xi\, dx=\int_\Gamma \, g_3\, dx',
\end{array}
\end{equation}
which follows from the double integration by parts
\begin{equation*}
\begin{array}{ll}\label{}
\ds \int_{\Omega^-}\Xi\, dx= \int_{\Omega^-}{\rm div}\,\dot{\mathfrak{h}}\, dx
=\int_\Gamma \, \dot{\mathfrak{h}}_1\, dx'
=\int_\Gamma \, \{ \partial_2\bigl(\widehat{\mathcal{H}}_2\varphi \bigr) +\partial_3\bigl(\widehat{\mathcal{H}}_3\varphi \bigr)+g_3 \}\, dx'
=\int_\Gamma \, g_3\, dx'.
\end{array}
\end{equation*}
We assume that the source terms $f, \chi,\Xi $ and the boundary datum $g$ vanish in the past and consider the case of zero initial data. We postpone the case of nonzero initial data to the nonlinear analysis (see e.g. \cite{CS,trakhinin09arma}).

%%%%%%%%%%%%%%%%%%%%%%%%%%%%%%%%%%%%%%%%%%%%%
%%%%%%%%%%%%%%%%%%%%%%%%%%%%%%%%%%%%%%%%%%%%%
%%%%%%%%%%%%%%%%%%%%%%%%%%%%%%%%%%%%%%%%%%%%%
\subsection{Reduction to homogeneous constraints in the \lq\lq vacuum part\rq\rq}

We decompose $\dot{\mathcal{H}}$ in \eqref{34} as $\dot{\mathcal{H}}=\mathcal{H}'+\mathcal{H}''$ (and accordingly $\dot{\mathfrak{H}}=\mathfrak{H}'+\mathfrak{H}''$, $\dot{\mathfrak{h}}=\mathfrak{h}'+\mathfrak{h}''$), where $\mathcal{H}''$ is required to solve for each $t$ the elliptic problem
\begin{equation}
\begin{array}{ll}\label{elliptic}
\nabla\times {\mathfrak{H}}''=\chi,\quad {\rm div}\,{\mathfrak{h}}''=\Xi \qquad &\mbox{in}\ \Omega^-, 
\\
{\mathfrak{h}}''_1={\mathcal{H}}_{N}'' =g_3\qquad &\mbox{on}\ \Gamma.
\end{array}
\end{equation}
The source term $ \chi$ of the first equation should satisfy the constraint ${\rm div}\, \chi=0$. For the resolution of  \eqref{elliptic} the data $\Xi,g_3$ must satisfy the necessary compatibility condition \eqref{compatibility}.
By classical results of the elliptic theory we have the following result.
\begin{lemma}\label{elliptic2}
Assume that the data $(\chi,\Xi, g_3)$ in \eqref{elliptic}, vanishing in appropriate way as $x$ goes to infinity, satisfy the constraint ${\rm div}\, \chi=0$ and the compatibility condition \eqref{compatibility}. Then there exists a unique solution $\mathcal{H}''$ of \eqref{elliptic} vanishing at infinity.
\end{lemma}
\begin{remark}\label{}
In the statement of the lemma above we intentionally leave unspecified the description of the regularity and the behavior at infinity of the data and consequently of the solution. This point will be faced in the forthcoming paper on the resolution of the nonlinear problem.
\end{remark}
Given $\mathcal{H}''$, now we look for $\mathcal{H}'$ such that
\begin{equation}
\begin{array}{ll}\label{postelliptic}
\nabla\times {\mathfrak{H}}'=0,\quad {\rm div}\,{\mathfrak{h}}'=0 \qquad &\mbox{in}\ Q^-_T, 
\\
{q}=\widehat{\mathcal{H}}\cdot{\mathcal{H}}'-  [ \partial_1\hat{q}] \varphi +g_2', \qquad & 
\\
{\mathcal{H}}_{N}' =\partial_2\bigl(\widehat{\mathcal{H}}_2\varphi \bigr) +\partial_3\bigl(\widehat{\mathcal{H}}_3\varphi \bigr)\qquad &\mbox{on}\ \omega_T,
\end{array}
\end{equation}
where we have denoted $g_2'=g_2+\widehat{\mathcal{H}}\cdot{\mathcal{H}}''$. If $\mathcal{H}''$ solves \eqref{elliptic} and $\mathcal{H}'$ is a solution of \eqref{postelliptic} then $\dot{\mathcal{H}}=\mathcal{H}'+\mathcal{H}''$ clearly solves \eqref{36}, \eqref{35}, \eqref{37}.

From \eqref{34}, \eqref{postelliptic}, the new form of the reduced linearized problem with unknowns ($U,\mathcal{H}'$) reads (we drop for convenience the $'$ in $\mathcal{H}',g_2'$)
\begin{subequations}\label{34''}
\begin{align}
\widehat{A}_0\partial_t\dot{U}+\sum_{j=1}^{3}\widehat{A}_j\partial_j\dot{U}+
\widehat{\mathcal C}\dot{U}=f \qquad &\mbox{in}\ Q^+_T,\label{34''a}
\\
\nabla\times {\mathfrak{H}}=0,\quad {\rm div}\,{\mathfrak{h}}=0 \qquad &\mbox{in}\ Q^-_T, \label{36''b}
\\
\partial_t\varphi=\dot{v}_{N}-\hat{v}_2\partial_2\varphi-\hat{v}_3\partial_3\varphi +
\varphi\,\partial_1\hat{v}_{N}+g_1,  \qquad &\label{37c}
\\
\dot{q}=\widehat{\mathcal{H}}\cdot{\mathcal{H}}-  [ \partial_1\hat{q}] \varphi +g_2, \qquad & \label{37d}
\\
{\mathcal{H}}_{N} =\partial_2\bigl(\widehat{\mathcal{H}}_2\varphi \bigr) +\partial_3\bigl(\widehat{\mathcal{H}}_3\varphi \bigr)\qquad &\mbox{on}\ \omega_T,
\label{37''}
\\
(\dot{U},{\mathcal{H}},\varphi )=0\qquad & \mbox{for}\ t<0.\label{38''f}
\end{align}
\end{subequations}
%

%%%%%%%%%%%%%%%%%%%%%%%%%%%%%%%%%%%%%%%%%%%%%%
%%%%%%%%%%%%%%%%%%%%%%%%%%%%%%%%%%%%%%%%%%%%%%
%%%%%%%%%%%%%%%%%%%%%%%%%%%%%%%%%%%%%%%%%%%%%%
%%%%%%%%%%%%%%%%%%%%%%%%%%%%%%%%%%%%%%%%%%%%%
\subsection{Reduction to homogeneous constraints in the \lq\lq plasma part\rq\rq}
From problem \eqref{34''} we can deduce nonhomogeneous equations associated with the divergence constraint ${\rm div}\,\dot h=0$ and the ``redundant'' boundary conditions $\dot H_N|_{x_1=0}=0$ for the nonlinear problem. More precisely, with reference to \cite[Proposition 2]{trakhinin09arma} for the proof, we have the following.

\begin{proposition}[\cite{trakhinin09arma}]
Let the basic state \eqref{21} satisfies assumptions \eqref{22}--\eqref{27}.
Then solutions of problem \eqref{34''} satisfy
\begin{equation}
{\rm div}\,\dot{h}=r\quad\mbox{in}\ Q^+_T,
\label{43}
\end{equation}
\begin{equation}
\widehat{H}_2\partial_2\varphi +\widehat{H}_3\partial_3\varphi -\dot{H}_{N}-
\varphi\,\partial_1\widehat{H}_{N}=R\quad\mbox{on}\ \omega_T.
\label{44}
\end{equation}
Here
\[
\dot{h}=(\dot{H}_{N},\dot{H}_2\partial_1\widehat{\Phi}_1,\dot{H}_3\partial_1\widehat{\Phi}_1),\quad
\dot{H}_{N}=\dot{H}_1-\dot{H}_2\partial_2\widehat{\Psi}-\dot{H}_3
\partial_3\widehat{\Psi}.
\]
The functions $r=
r(t,x )$ and $R= R(t,x')$, which vanish in the past, are determined by the source terms and the basic state as solutions to the linear inhomogeneous equations
\begin{equation}
\partial_t a+ \frac{1}{\partial_1\widehat{\Phi}_1}\left\{ \hat{w} \cdot\nabla a + a\,{\rm div}\,\hat{u}\right\}={\mathcal F}_H\quad\mbox{in}\ Q^+_T,
\label{45}
\end{equation}
\begin{equation}
\partial_t R +\hat{v}_2\partial_2R+\hat{v}_3\partial_3R+
\left(\partial_2\hat{v}_2+\partial_3\hat{v}_3\right) R={\mathcal Q}\quad \mbox{on}\ \omega_T,
\label{46}
\end{equation}
where $a=r/\partial_1\widehat{\Phi}_1,\quad {\mathcal F}_H=({\rm div}\,
{f}_{H})/\partial_1\widehat{\Phi}_1$,
\[
{f}_{H}=
(f_{N} ,f_6,f_7),\quad
f_{N}=f_5-f_6\partial_2\widehat{\Psi}-
f_7\partial_3\widehat{\Psi},
\quad
{\mathcal Q}=\bigl\{\partial_2\bigl(\widehat{H}_2g_1\bigr)+
\partial_3\bigl(\widehat{H}_3g_1\bigr)-f_{N}\bigr\}\bigr|_{x_1=0}.
\]
\label{p3.1}
\end{proposition}
%%%%%%%%%%%%%%%%%%%%%%%%%%%%%%%%%%%%%%%%%%%%%
%%%%%%%%%%%%%%%%%%%%%%%%%%%%%%%%%%%%%%%%%%%%%%
Let us reduce \eqref{34''} to a problem with homogeneous boundary conditions \eqref{37c}, \eqref{37d} (i.e. $g_1=g_2=0$) and homogeneous constraints \eqref{43} and \eqref{44} (i.e. $r=R=0$). More precisely, we describe a ``lifting'' function as follows:
\begin{equation*}
\widetilde{U} = (\tilde{q},\tilde{v}_1,0,0,\widetilde{H},0),\qquad
\label{tilde}
\end{equation*}
where
$\tilde{q}= {g}_2,\tilde{v}_1=-g_1$ on $\omega_T$, and where
$\widetilde{H}$ solves the equation for $\dot{H}$ contained in \eqref{34''a} with $\dot{v}=0$:
\begin{equation}
\partial_t\widetilde{H}+ \frac{1}{\partial_1\widehat{\Phi}_1}\left\{(\hat{w} \cdot\nabla )
\widetilde{H} - (\tilde{h} \cdot\nabla ) \hat{v} + \widetilde{H}{\rm div}\,\hat{w}\right\}
 ={f}_{H} \qquad \mbox{in}\ Q^+_T,
 \label{42}
\end{equation}
where $\tilde{h}=(\widetilde{H}_1-\widetilde{H}_2\partial_2\hat{\Psi}
-\widetilde{H}_3\partial_3\hat{\Psi},\widetilde{H}_2,\widetilde{H}_3)$, $f_H=(f_5,f_6,f_ 7)$. It is very important that, in view of \eqref{24}, we have $\hat{w}_1|_{x_1=0}=0$; therefore the linear equation \eqref{42} does not need any boundary condition. 
Then the new unknown
\begin{equation}
U^{\natural}=\dot{U}-\widetilde{U},\quad \mathcal{H}^{\natural}={\mathcal{H}}\label{87'}
\end{equation}
satisfies problem \eqref{34''} with $f=F$, where
\[
F =(F_1,\ldots, F_8)=f-\mathbb{P}'_e(\widehat{U},\widehat{\Psi})\widetilde{U}.
\]
In view of \eqref{42}, $F_H=(F_5,F_6,F_7)=0$, and it follows from Proposition \ref{p3.1} that $U^{\natural}$ satisfies \eqref{43} and \eqref{44} with $r=R=0$. 

Dropping for convenience the indices $^{\natural}$ in \eqref{87'}, the new form of our reduced linearized problem now reads
\begin{subequations}\label{34'}
\begin{align}
\widehat{A}_0\partial_t{U}+\sum_{j=1}^{3}\widehat{A}_j\partial_j{U}+
\widehat{\mathcal C}{U}=F \qquad &\mbox{in}\ Q^+_T,\label{34'a}
\\
\nabla\times {\mathfrak{H}}=0,\quad {\rm div}\,{\mathfrak{h}}=0 \qquad &\mbox{in}\ Q^-_T, \label{36'b}
\\
\partial_t\varphi={v}_{N}-\hat{v}_2\partial_2\varphi-\hat{v}_3\partial_3\varphi +
\varphi\,\partial_1\hat{v}_{N},  \qquad &
\\
{q}=\widehat{\mathcal{H}}\cdot{\mathcal{H}}-  [ \partial_1\hat{q}] \varphi , \qquad & \label{35'd}
\\
{\mathcal{H}}_{N} =\partial_2\bigl(\widehat{\mathcal{H}}_2\varphi \bigr) +\partial_3\bigl(\widehat{\mathcal{H}}_3\varphi \bigr)\qquad &\mbox{on}\ \omega_T,
\label{37'}
\\
({U},{\mathcal{H}},\varphi )=0\qquad & \mbox{for}\ t<0.\label{38'f}
\end{align}
\end{subequations}
and solutions should satisfy
\begin{equation}
{\rm div}\,{h}=0\qquad\mbox{in}\ Q^+_T,
\label{93}
\end{equation}
\begin{equation}
{H}_{N}=\widehat{H}_2\partial_2\varphi +\widehat{H}_3\partial_3\varphi -
\varphi\,\partial_1\widehat{H}_{N}\quad\mbox{on}\ \omega_T.
\label{95}
\end{equation}
All the notations here for $U$ and $\mathcal{H}$ (e.g., $h$, $\mathfrak{H}$, $\mathfrak{h}$, etc.) are analogous to the corresponding ones for $\dot{U}$ and $\dot{\mathcal{H}}$ introduced above.

%%%%%%%%%%%%%%%%%%%%%%%%%%%%%%%%%%%%%%%%%%%%%%%%%%%%%%
%%%%%%%%%%%%%%%%%%%%%%%%%%%%%%%%%%%%%%%%%%%%%%%%%%%%%%
%%%%%%%%%%%%%%%%%%%%%%%%%%%%%%%%%%%%%%%%%%%%%%%%%%%%%%
\subsection{An equivalent formulation of \eqref{34'} }\label{equiva}
In the following analysis it is convenient to make use of different \lq\lq plasma\rq\rq variables and an equivalent form of equations \eqref{34'a}.
We define the matrix
\begin{equation*}
\begin{array}{ll}\label{}
\hat \eta=\begin{pmatrix}
 1&-\ddue\widehat\Psi &-\dtre\widehat\Psi \\
0 &\duno\widehat\Phi_1&0\\
0&0&\duno\widehat\Phi_1
\end{pmatrix}.
\end{array}
\end{equation*}
It follows that
\begin{equation}
\begin{array}{ll}\label{defcalU}
{u}=({v}_{N},{v}_2\partial_1\widehat{\Phi}_1,{v}_3\partial_1\widehat{\Phi}_1)=\hat \eta\, v, \qquad
{h}=({H}_{N},{H}_2\partial_1\widehat{\Phi}_1,{H}_3\partial_1\widehat{\Phi}_1)=\hat\eta \,H.
\end{array}
\end{equation}
Multiplying \eqref{34'a} on the left side by the matrix
\begin{equation*}
\begin{array}{ll}\label{}
\widehat R=\begin{pmatrix}
 1&\underline 0&\underline 0&0 \\
 \underline0^T&\hat \eta&0_3& \underline 0^T\\
 \underline 0^T&0_3&\hat \eta&\underline 0^T\\
 0&\underline 0^T&\underline 0^T&1
\end{pmatrix},
\end{array}
\end{equation*}
after some calculations we get the symmetric hyperbolic system for the new vector of unknowns $\mathcal{U}=(q,u,h,S)$
(compare with \eqref{3'}, \eqref{34'a}):
\begin{equation}
\begin{array}{ll}\label{34'''}
\duno\widehat\Phi_1
\left(\begin{matrix}
{\hat\rho_p/\hat\rho}&\underline
0&-({\hat\rho_p/\hat\rho})\hat h &0 \\
\underline 0^T&\hat\rho
\hat a_0&0_3&\underline 0^T\\
-({\hat\rho_p/\hat\rho})\hat h^T&0_3&\hat a_0 +({\hat\rho_p/\hat\rho})\hat h\otimes\hat h&\underline 0^T\\
0&\underline 0&\underline 0&1
\end{matrix}\right)\dt
\left(\begin{matrix}
q \\ u \\
h\\S \end{matrix}\right)
 +
\left( \begin{matrix}
0&\nabla\cdot&\underline 0&0\\
\nabla&0_3&0_3 &\underline 0^T\\
\underline 0^T&0_3 &0_3&\underline 0^T\\
0&\underline 0&\underline 0&0
\end{matrix}\right)
\left(\begin{matrix}q \\ u \\ h\\S \end{matrix}\right)
\\
 \\
 +
\duno\widehat\Phi_1
\left( \begin{matrix}
(\hat\rho_p/\hat\rho)
\hat w \cdot\nabla&\nabla\cdot&-({\hat\rho_p/\hat\rho})\hat h\hat w \cdot\nabla&0\\
\nabla&\hat\rho \hat a_0\hat w \cdot\nabla&-\hat a_0\hat h \cdot\nabla &\underline 0^T\\
-({\hat\rho_p/\hat\rho})\hat h^T \hat w \cdot\nabla&-\hat a_0\hat h \cdot\nabla &(\hat a_0 +({\hat\rho_p/\hat\rho})\hat h\otimes\hat h)
\hat w \cdot\nabla&\underline 0^T\\
0&\underline 0&\underline 0&\hat w \cdot\nabla
\end{matrix}\right)
\left(\begin{matrix}q \\ u \\ h\\S \end{matrix}\right)
+\widehat{\mathcal{C}}'\mathcal{U}=\mathcal{F}\,,
\end{array}
\end{equation}
where $\hat a_0$ is the symmetric and positive definite matrix
$$
\hat a_0 =(\hat \eta^{-1})^T\hat \eta^{-1},$$
with a new matrix $\widehat{\mathcal{C}}'$ in the zero-order term (whose precise form has no importance) and where we have set
$
\mathcal{F}=\duno\widehat\Phi_1 \, \widehat R
F.$
We write system \eqref{34'''} in compact form as
\begin{equation}
\begin{array}{ll}\label{}
\displaystyle \widehat{\mathcal{A}}_0\partial_t{\mathcal{U}}+\sum_{j=1}^{3}(\widehat{\mathcal{A}}_j+{\mathcal{E}}_{1j+1})\partial_j{\mathcal{U}}+
\widehat{\mathcal C}'{\mathcal{U}}=\mathcal{F} ,\label{73}
\end{array}
\end{equation}
where
\[
\mathcal{E}_{12}=\left(\begin{array}{cccccc}
0& 1 &0 &0 & \cdots & 0 \\
1 & 0 &0 &0 & \cdots & 0 \\
0 & 0 &0 &0 & \cdots & 0 \\
0 & 0 &0 &0 & \cdots & 0 \\
\vdots & \vdots &\vdots &\vdots& & \vdots \\
0& 0 &0 &0 & \cdots & 0  \end{array}
 \right), \qquad
\,
\mathcal{E}_{13}=\left(\begin{array}{cccccc}
0& 0 &1 &0 & \cdots & 0 \\
0 & 0 &0 &0 & \cdots & 0 \\
1 & 0 &0 &0 & \cdots & 0 \\
0 & 0 &0 &0 & \cdots & 0 \\
\vdots & \vdots &\vdots &\vdots& & \vdots \\
0& 0 &0 &0 & \cdots & 0  \end{array}
 \right),
 \]
\[
\mathcal{E}_{14}=\left(\begin{array}{cccccc}
0& 0 &0 & 1&\cdots & 0 \\
0 & 0 &0 &0& \cdots & 0 \\
0 & 0 &0 &0& \cdots & 0 \\
1 & 0 &0 &0& \cdots & 0 \\
\vdots &\vdots & \vdots &\vdots & & \vdots \\
0& 0 &0 &0& \cdots & 0
\end{array}
 \right).
\]
The formulation \eqref{73} has the advantage of the form of the boundary matrix of the system $\widehat{\mathcal{A}}_1+{\mathcal{E}}_{12}$, with
\begin{equation}
\begin{array}{ll}\label{a10}
\widehat{\mathcal{A}}_1=0 \qquad\mbox{on }\omega_T,

\end{array}
\end{equation}
because $\hat w_1=\hat h_1=0$, and ${\mathcal{E}}_{12}$ a constant matrix.
Thus system \eqref{73} is symmetric hyperbolic with characteristic boundary of constant multiplicity (see \cite{rauch85,secchi95,secchi96} for maximally dissipative boundary conditions). Thus, the final form of our reduced linearized problem is
\begin{subequations}\label{34'new}
\begin{align}
\displaystyle \widehat{\mathcal{A}}_0\partial_t{\mathcal{U}}+\sum_{j=1}^{3}(\widehat{\mathcal{A}}_j+{\mathcal{E}}_{1j+1})\partial_j{\mathcal{U}}+
\widehat{\mathcal C}'{\mathcal{U}}=\mathcal{F} ,
 \qquad &\mbox{in}\ Q^+_T,\label{34'anew}
\\
\nabla\times {\mathfrak{H}}=0,\quad {\rm div}\,{\mathfrak{h}}=0 \qquad &\mbox{in}\ Q^-_T, \label{36'bnew}
\\
\partial_t\varphi=u_1-\hat{v}_2\partial_2\varphi-\hat{v}_3\partial_3\varphi +
\varphi\,\partial_1\hat{v}_{N},  \qquad &
\\
{q}=\widehat{\mathcal{H}}\cdot{\mathcal{H}}-  [ \partial_1\hat{q}] \varphi , \qquad & \label{35'dnew}
\\
{\mathcal{H}}_{N} =\partial_2\bigl(\widehat{\mathcal{H}}_2\varphi \bigr) +\partial_3\bigl(\widehat{\mathcal{H}}_3\varphi \bigr)\qquad &\mbox{on}\ \omega_T,
\label{37'new}
\\
(\mathcal{U},{\mathcal{H}},\varphi )=0\qquad & \mbox{for}\ t<0,\label{38'fnew}
\end{align}
\end{subequations}
under the constraints \eqref{93}, \eqref{95}.
%%%%%%%%%%%%%%%%%%%%%%%%%%%%%%%%%%%%%%%%%%%%%%
%%%%%%%%%%%%%%%%%%%%%%%%%%%%%%%%%%%%%%%%%%%%%%
\section{Function Spaces}\label{fs}
Now we introduce the main function spaces to be used in the following.
Let us denote
\begin{equation}
\begin{array}{ll}\label{defQ}
Q^\pm:= \R_t\times\Omega^\pm,\quad \omega:=\R_t\times\Gamma.
\end{array}
\end{equation}

\subsection{Weighted Sobolev spaces}
For $\gamma\ge 1$ and $s\in\mathbb{R}$, we set
\begin{equation*}\label{weightfcts}
\lambda^{s,\gamma}(\xi):=(\gamma^2+|\xi|^2)^{s/2}
\end{equation*}
and, in particular, $\lambda^{s,1}:=\lambda^{s}$.
\newline
Throughout the paper, for real $\gamma\ge 1$ and $n\ge2$, $H^s_{\gamma}(\mathbb{R}^n)$ will denote the Sobolev space of order $s$, equipped with the $\gamma-$depending norm $||\cdot||_{s,\gamma}$ defined by
\begin{equation}\label{normagamma}
||u||^2_{s,\gamma}:=(2\pi)^{-n}\int_{\mathbb{R}^n}\lambda^{2s,\gamma}(\xi)|\widehat{u}(\xi)|^2d\xi\,,
\end{equation}
$\widehat{u}$ being the Fourier transform of $u$. The norms defined by \eqref{normagamma}, with different values of the parameter $\gamma$, are equivalent each other. For $\gamma=1$ we set for brevity $||\cdot||_{s}:=||\cdot||_{s,1}$ (and, accordingly, the standard Sobolev space $H^s(\mathbb{R}^n):=H^s_{1}(\mathbb{R}^n)$).
For $s\in\mathbb{N}$, the norm in \eqref{normagamma} turns to be equivalent, {\it uniformly with respect to} $\gamma$, to the norm $||\cdot||_{H^s_{\gamma}(\mathbb{R}^n)}$ defined by
\begin{equation*}\label{derivate}
||u||^2_{H^s_{\gamma}(\mathbb{R}^n)}:=\sum\limits_{|\alpha|\le s}\gamma^{2(s-|\alpha|)}||\partial^{\alpha}u||^2_{L^2(\mathbb{R}^n)}\,.
\end{equation*}
\noindent
For functions defined over $Q^-_T$ we will consider the weighted Sobolev spaces $H^m_{\gamma}(Q^-_T)$ equipped with the $\gamma-$depending norm
\begin{equation*}\label{derivate2}
||u||^2_{H^m_{\gamma}(Q^-_T)}:=\sum\limits_{|\alpha|\le m}\gamma^{2(m-|\alpha|)}||\partial^{\alpha}u||^2_{L^2(Q^-_T)}\,.
\end{equation*}
Similar weighted Sobolev spaces will be considered for functions defined on $Q^-$.
%%%%%%%%%%%%%%%%%%%%%%%%%%%%%%%%%%%%

\subsection{Conormal Sobolev spaces}
Let us introduce some classes of function spaces of Sobolev type, defined over the half-space $Q^+_T$.
For $j=0,\dots,3$, we set
$$
Z_0=\dt,\quad Z_1:=\sigma (x_1)\partial_1\,,\quad Z_j:=\partial_j\,,\,\,{\rm for}\,\,j= 2,3\,,
$$
where $\sigma (x_1)\in C^{\infty}(\mathbb{R}_+)$ is
a monotone increasing function such that $\sigma (x_1)=x_1$ in a neighborhood of
the origin and $\sigma (x_1)=1$ for $x_1$ large enough.
Then, for every multi-index $\alpha=(\alpha_0,\dots,\alpha_3)\in\mathbb{N}^4$, the {\it conormal} derivative $Z^{\alpha}$ is defined by
$$
Z^{\alpha}:=Z_0^{\alpha_0}\dots Z^{\alpha_3}_3\,;
$$
we also write $\partial^{\alpha}=\partial^{\alpha_0}_0\dots\partial^{\alpha_3}_3$ for the usual partial derivative corresponding to $\alpha$.
\newline
Given an integer $m\geq 1$, the {\it conormal Sobolev space} $H^m_{tan}(Q^+_T)$ is defined as the set of functions $u\in L^2(Q^+_T)$ such that
$Z^\alpha u\in  L^2(Q^+_T)$, for all multi-indices $\alpha$ with $|\alpha|\le m$ (see \cite{moseBVP,moseIBVP}). Agreeing with the notations set for the usual Sobolev spaces, for $\gamma\ge 1$, $H^m_{tan,\gamma}(Q^+_T)$ will denote the conormal space of order $m$ equipped with the $\gamma-$depending norm
\begin{equation}\label{normaconormale}
||u||^2_{H^{m}_{tan,\gamma}(Q^+_T)}:=\sum\limits_{|\alpha|\le m}\gamma^{2(m-|\alpha|)}||Z^{\alpha}u||^2_{L^2(Q^+_T)}\,
\end{equation}
\noindent
and we have $H^m_{tan}(Q^+_T):=H^m_{tan,1}(Q^+_T)$. Similar conormal Sobolev spaces with $\gamma$-depending norms will be considered for functions defined on $Q^+$.

We will use the same notation for spaces of scalar and vector-valued functions.
%%%%%%%%%%%%%%%%%%%%%%%%%%%%%%%%%%%%%%%%%%%%%%
%%%%%%%%%%%%%%%%%%%%%%%%%%%%%%%%%%%%%%%%%%%%%%
\section{The main result}
%%%%%%%%%%%%%%%%%%%%%%%%%%%%%%%%%%%%%%%%%%%%%%
We are now in a position to state the main result of this paper. Recall that $\mathcal{U}=(q,u,h,S)$, where $u$ and $h$ were defined in \eqref{defcalU}.

\begin{theorem}\label{main}
Let $T>0$. Let the basic state \eqref{21} satisfies assumptions \eqref{22}--\eqref{27} and
\begin{equation}
|\widehat{H} \times \widehat{\mathcal{H}}|\geq \delta > 0 \qquad \mbox{on  }\omega_T,\label{41}
\end{equation}
where $\delta$ is a fixed constant. There exists $\gamma_0\ge1$ such that for all $\gamma\ge\gamma_0$ and for all $\mathcal{F}_\gamma \in H^1_{tan,\gamma}(Q^+_T)$, vanishing in the past, namely for $t<0$, problem \eqref{34'new} has a unique solution $(\mathcal{U}_\gamma,{\mathcal{H}}_\gamma,\varphi_\gamma)\in H^1_{tan,\gamma}(Q^+_T)\times H^1_{\gamma}(Q^-_T)\times H^1_\gamma(\omega_T)$ with trace $(q_\gamma,u_{1\gamma},h_{1\gamma})|_{\omega_T}\in {H^{1/2}_\gamma(\omega_T)}$, $\mathcal{H}_\gamma|_{\omega_T}\in {H^{1/2}_\gamma(\omega_T)}$.
Moreover, the solution obeys the a priori estimate
\begin{multline}
\gamma\left(\|\mathcal{U}_\gamma\|^2_{H^1_{tan,\gamma}(Q^+_T)}+\|{\mathcal{H}}_\gamma\|^2_{H^{1}_\gamma(Q^-_T)}
+\|(q_\gamma,u_{1\gamma},h_{1\gamma})|_{\omega_T}\|^2_{H^{1/2}_\gamma(\omega_T)}
+\|\mathcal{H}_\gamma|_{\omega_T}\|^2_{H^{1/2}_\gamma(\omega_T)}\right)
\\
+\gamma^2\|\varphi_\gamma\|^2_{H^1_\gamma(\omega_T)}
\leq  \frac{C}{\gamma}\|\mathcal{F}_\gamma\|^2_{H^1_{tan,\gamma}(Q^+_T)},
\label{54}
\end{multline}
where we have set $\mathcal{U}_\gamma=e^{-\gamma t}\,\mathcal{U}, \mathcal{H}_\gamma=e^{-\gamma t}\, \mathcal{H}, \varphi_\gamma= e^{-\gamma t}\, \varphi$
and so on. Here $C=C(K,T,\delta)>0$ is a constant independent of the data $\mathcal{F}$ and $\gamma$.
\label{t1}
\end{theorem}
The a priori estimate \eqref{54} improves the similar estimate firstly proved in \cite{trakhinin10}.

\begin{remark}\label{}
Strictly speaking, the uniqueness of the solution to problem \eqref{34'new} follows from the a priori estimate (42) derived in \cite{trakhinin10}, provided that our solution belongs to $H^2$. We do not present here a formal proof of the existence of solutions with a higher degree of regularity (in particular, $H^2$) and postpone this part to the future work on the nonlinear problem (see e.g. \cite{CS,trakhinin09arma}). 

\end{remark}
The remainder of the paper is organized as follows. In the next Section \ref{s3} we introduce a fully hyperbolic regularization of the coupled hyperbolic-elliptic system \eqref{34'new}. In Section \ref{basic} we show an a priori estimate of solutions uniform in the small parameter $\varepsilon$ of regularization. In Section \ref{well} we show the well-posedness of the hyperbolic regularization and in Section \ref{proofmain} we conclude the proof of Theorem \ref{main} by passing to the limit as $\varepsilon\to0$. Sections \ref{equival}, \ref{prooflemma1}, \ref{proofinterpol} are devoted to the proof of some technical results. 
%%%%%%%%%%%%%%%%%%%%%%%%%%%%%%%%%%%%%%%%%%%%%%
%%%%%%%%%%%%%%%%%%%%%%%%%%%%%%%%%%%%%%%%%%%%%%
%%%%%%%%%%%%%%%%%%%%%%%%%%%%%%%%%%%%%%%%%%%%%%
%%%%%%%%%%%%%%%%%%%%%%%%%%%%%%%%%%%%%%%%%%%%%%
%%%%%%%%%%%%%%%%%%%%%%%%%%%%%%%%%%%%%%%%%%%%%%

\section{Hyperbolic regularization of the reduced problem}
\label{s3}

The problem \eqref{34'new} is a nonstandard initial-boundary value problem for a coupled hyperbolic-elliptic system. For its resolution we introduce a ``hyperbolic'' regularization of the elliptic system \eqref{36'bnew}. We will prove the existence of solutions for such regularized problem by referring to the well-posedness theory for linear symmetric hyperbolic systems with characteristic boundary and maximally nonnegative boundary conditions \cite{secchi95,secchi96}. After showing suitable a priori estimate uniform in $\varepsilon$, we will pass to the limit as $\varepsilon\to0$, to get the solution of \eqref{34'new}.

The regularization of problem \eqref{34'new} is inspired by a corresponding problem in relativistic MHD \cite{trakhinin10a}. In our non-relativistic case the displacement current $(1/c)\partial_tE$ is neglected in the vacuum Maxwell equations, where $c$ is the speed of light and $E$ is the electric field. Now, in some sense, we restore this neglected term. Namely, we consider a ``hyperbolic'' regularization of the elliptic system \eqref{36'bnew} by introducing a new auxiliary unknown $E^{\varepsilon}$ which plays a role of the vacuum electric field, and the small parameter of regularization $\varepsilon$ is associated with the physical parameter $1/c$. We also regularize the second boundary condition in \eqref{35'dnew} and introduce two boundary conditions for the unknown $E^{\varepsilon}$.

Let us denote $V^{\varepsilon}=(\mathcal{H}^{\varepsilon},E^{\varepsilon})$. Given a small parameter $\varepsilon>0$, we consider the following regularized problem for the unknown $(\mathcal{U}^{\varepsilon},V^{\varepsilon},\varphi^{\varepsilon})$:
%%%%%%%%%%%%%%%%%%%%%%%%%%%%%%%%%%%%%%%%%
\begin{subequations}\label{34"}
\begin{align}
\displaystyle \widehat{\mathcal{A}}_0\partial_t{\mathcal{U}}^{\varepsilon}+\sum_{j=1}^{3}(\widehat{\mathcal{A}}_j+{\mathcal{E}}_{1j+1})\partial_j{\mathcal{U}}^{\varepsilon}+
\widehat{\mathcal C}'{\mathcal{U}}^{\varepsilon}=\mathcal{F} 
 \qquad &\mbox{in}\ Q^+_T,\label{34"a}
\\
\varepsilon\partial_t\mathfrak{h}^{\varepsilon}+\nabla\times \mathfrak{E}^{\varepsilon}=0, \qquad
\varepsilon\partial_t\mathfrak{e}^{\varepsilon}-\nabla\times \mathfrak{H}^{\varepsilon}=0  \qquad &\mbox{in}\ Q^-_T, \label{35"}
\\
\partial_t\varphi^{\varepsilon}={u}_{1}^{\varepsilon}-\hat{v}_2\partial_2\varphi^{\varepsilon}-\hat{v}_3\partial_3\varphi^{\varepsilon} +
\varphi^{\varepsilon}\partial_1\hat{v}_{N}, \qquad&\label{36"} \\
{q}^{\varepsilon}=\widehat{\mathcal{H}}\cdot{\mathcal{H}}^{\varepsilon}-  [ \partial_1\hat{q}]\varphi^{\varepsilon} -\varepsilon\,
\widehat{E}\cdot E^{\varepsilon},  \qquad&\label{}
\\
{E}_{\tau_2}^{\varepsilon}=\varepsilon\,\partial_t(\widehat{\mathcal{H}}_3\varphi^{\varepsilon} )-\varepsilon\,\partial_2(\widehat{E}_1\varphi^{\varepsilon} ), \qquad&\label{51e} \\
{E}_{\tau_3}^{\varepsilon}=-\varepsilon\,\partial_t(\widehat{\mathcal{H}}_2\varphi^{\varepsilon} )-\varepsilon\,\partial_3(\widehat{E}_1\varphi^{\varepsilon} ) 
 \qquad&\mbox{on}\ \omega_T, \label{37"}
\\
(\mathcal{U}^{\varepsilon},V^{\varepsilon},\varphi^{\varepsilon} )=0\qquad &\mbox{for}\ t<0,\label{38"g}
\end{align}
\end{subequations}
%%%%
where
\[
E^{\varepsilon}=(E_1^{\varepsilon},E_2^{\varepsilon},E_3^{\varepsilon}),\quad \widehat{E}=(\widehat{E}_1,\widehat{E}_2,\widehat{E}_3),\quad
{\mathfrak{E}}^{\varepsilon}=(E_1^{\varepsilon}\partial_1\widehat{\Phi}_1,
E_{\tau_2}^{\varepsilon},E_{\tau_3}^{\varepsilon}),
\]
\[
{\mathfrak{e}}^{\varepsilon}=(E_{N}^{\varepsilon},E_2^{\varepsilon}\partial_1\widehat{\Phi}_1,
E_3^{\varepsilon}\partial_1\widehat{\Phi}_1),\quad
E_{N}^{\varepsilon}=E_1^{\varepsilon}-E_2^{\varepsilon}\partial_2\widehat{\Psi}-E_3^{\varepsilon}\partial_3\widehat{\Psi},\quad
E_{\tau_k}^{\varepsilon}=E_1^{\varepsilon}\partial_k\widehat{\Psi}+E_k^{\varepsilon},\ k=2,3,
\]
the coefficients $\widehat{E}_j$ are given functions which will be chosen later on. All the other notations for $\mathcal{H}^{\varepsilon}$ (e.g., $\mathfrak{H}^{\varepsilon}$, $\mathfrak{h}^{\varepsilon}$) are analogous to those for $\mathcal{H}$. 
%%%

%Differently from \eqref{34'anew}, the boundary is noncharacteristic for the above system \eqref{34"a} because the boundary matrix ${\varepsilon}I-{\mathcal{E}}_{12}$ is invertible for small $\varepsilon$.
%
If $\Psi=0, \Phi_1=x_1$, then ${\mathfrak{h}}^{\varepsilon}={\mathfrak{H}}^{\varepsilon}={\mathcal{H}}^{\varepsilon}\,, {\mathfrak{e}}^{\varepsilon}={\mathfrak{E}}^{\varepsilon}={{E}}^{\varepsilon}$, and when $\varepsilon=1$ \eqref{35"} 
%becomes a noncharacteristic regularization of 
turns out to be nothing else than the Maxwell equations.

It is noteworthy that solutions to problem \eqref{34"} satisfy
%%%%
\begin{eqnarray}
{\rm div}\,{h}^{\varepsilon}=0\qquad&\mbox{in}\ Q^+_T,
\label{93"}
\\
{\rm div}\,{\mathfrak{h}}^{\varepsilon}=0,\quad {\rm div}\,{\mathfrak{e}}^{\varepsilon}=0\qquad&\mbox{in}\ Q^-_T,
\label{94"}
\\
{h}_{1}^{\varepsilon}=\widehat{H}_2\partial_2\varphi^{\varepsilon} +\widehat{H}_3\partial_3\varphi^{\varepsilon} -
\varphi^{\varepsilon}\partial_1\widehat{H}_{N},\qquad&
\label{95"}
\\
{\mathcal{H}}_{N}^{\varepsilon} =\partial_2\bigl(\widehat{\mathcal{H}}_2\varphi^{\varepsilon} \bigr) +\partial_3\bigl(\widehat{\mathcal{H}}_3\varphi^{\varepsilon} \bigr)  \qquad &\mbox{on}\ \omega_T,
\label{96"}
\end{eqnarray}
%%%
because \eqref{93"}--\eqref{96"} are just restrictions on the initial data which are automatically satisfied in view of \eqref{38"g}.
Indeed, the derivation of \eqref{93"} and \eqref{95"} is absolutely the same as that of \eqref{93} and \eqref{95}. Equations \eqref{94"} trivially follow from \eqref{35"}, \eqref{38"g}. Moreover, condition \eqref{96"} is obtained by considering the first component of the first equation in \eqref{35"} at $x_1=0$ and taking into account \eqref{51e} - \eqref{38"g}.

%%%%%%%%%%%%%%%%%%%%%%%%%%%%%%%%%%%%%%%%%
%%%%%%%%%%%%%%%%%%%%%%%%%%%%%%%%%%%%%%%%%
\subsection{An equivalent formulation of \eqref{34"}}\label{equiva2}

In the following analysis it is convenient to make use of a different formulation of the approximating problem \eqref{34"}, as far as the vacuum part is concerned.

First we introduce the matrices which are coefficients of the space derivatives in \eqref{35"} (for $\varepsilon=1$ the matrices below are those for the vacuum Maxwell equations):
\[
B_1^{\varepsilon}=\varepsilon^{-1}\left(\begin{array}{cccccc}
0 & 0 & 0& 0 & 0 & 0 \\
0 & 0 & 0& 0 & 0 & -1 \\
0 & 0 & 0& 0 & 1 & 0 \\
0 & 0 & 0& 0 & 0 & 0 \\
0 & 0 & 1& 0 & 0 & 0 \\
0 & -1 & 0& 0 & 0 & 0
\end{array} \right),\quad
B_2^{\varepsilon}=\varepsilon^{-1}\left(\begin{array}{cccccc}
0 & 0 & 0& 0 & 0 & 1 \\
0 & 0 & 0& 0 & 0 & 0 \\
0 & 0 & 0& -1 & 0 & 0 \\
0 & 0 & -1& 0 & 0 & 0 \\
0 & 0 & 0& 0 & 0 & 0 \\
1 & 0 & 0& 0 & 0 & 0
\end{array} \right),
\]
\[
B_3^{\varepsilon}=\varepsilon^{-1}\left(\begin{array}{cccccc}
0 & 0 & 0& 0 & -1 & 0 \\
0 & 0 & 0& 1 & 0 & 0 \\
0 & 0 & 0& 0 & 0 & 0 \\
0 & 1 & 0& 0 & 0 & 0 \\
-1 & 0 & 0& 0 & 0 & 0 \\
0 & 0 & 0& 0 & 0 & 0
\end{array} \right).
\]
Then system \eqref{35"} can be written in terms of
 the ``curved'' unknown $W^{\varepsilon}=(\mathfrak{H}^{\varepsilon},\mathfrak{E}^{\varepsilon})$ as
\begin{equation}
{B}_0\partial_tW^{\varepsilon}+\sum_{j=1}^3B_j^{\varepsilon}\partial_jW^{\varepsilon} +{B}_4W^{\varepsilon}=0,\label{newmaxwell}
\end{equation}
where
\[
{B}_0=({\partial_1\widehat{\Phi}_1})^{-1}\,KK^{\textsf{T}}>0,\qquad
K=I_2\otimes\hat \eta,\qquad {B}_4=\partial_t{B}_0,
\]
and the matrices $B_0$ and $K$ are found from the relations
\[
\mathfrak{h}^{\varepsilon}=\hat\eta\, \mathcal{H}^{\varepsilon}=({\partial_1\widehat{\Phi}_1})^{-1}\hat\eta\, \hat\eta^T\mathfrak{H}^{\varepsilon}
,\qquad
\mathfrak{e}^{\varepsilon}=\hat\eta\,{E}^{\varepsilon}=({\partial_1\widehat{\Phi}_1})^{-1}\hat\eta\, \hat\eta^T\mathfrak{E}^{\varepsilon}
,\]
so that
$$
\begin{pmatrix}
\mathfrak{h}^{\varepsilon} \\
\mathfrak{e}^{\varepsilon}
\end{pmatrix}=({\partial_1\widehat{\Phi}_1})^{-1}
\begin{pmatrix}
\hat\eta\, \hat\eta^T &0_3 \\
 0_3& \hat\eta\, \hat\eta^T
\end{pmatrix}
\begin{pmatrix}
\mathfrak{H}^{\varepsilon} \\
\mathfrak{E}^{\varepsilon}
\end{pmatrix}=B_0W^{\varepsilon} .
$$
System \eqref{newmaxwell} is symmetric hyperbolic. The convenience of the use of variables $(\mathfrak{H}^{\varepsilon},\mathfrak{E}^{\varepsilon})$ rather than $(\mathcal{H}^{\varepsilon},{E}^{\varepsilon})$ stays mainly in that the matrices $B_j^{\varepsilon}$ of \eqref{newmaxwell}, containing the singular multiplier $\varepsilon^{-1}$, are constant.

%
%%%%%%%%%%%%%%%%%%%%%%%%%%%%%%
Finally, we write the boundary conditions \eqref{36"}--\eqref{37"} in terms of $(\mathcal{U}^{\varepsilon}, W^{\varepsilon})$, where we observe that (recalling that ${\partial_1\widehat{\Phi}_1}=1$ on $\omega_T$):
\begin{equation}
\begin{array}{ll}\label{nuovebc}
 \widehat{\mathcal{H}}\cdot{\mathcal{H}}^{\varepsilon} = \widehat{\mathcal{H}}_N{\mathcal{H}}^{\varepsilon}_1+\widehat{\mathcal{H}}_2{\mathcal{H}}^{\varepsilon}_{\tau_2}+\widehat{\mathcal{H}}_3{\mathcal{H}}^{\varepsilon}_{\tau_3}
=\hat{\mathfrak{h}}\cdot\mathfrak{H}^{\varepsilon},\\
\\
 \widehat{E}\cdot E^{\varepsilon}
= \widehat{E}_N{E}^{\varepsilon}_1+\widehat{E}_2{E}^{\varepsilon}_{\tau_2}+\widehat{E}_3{E}^{\varepsilon}_{\tau_3}
=\hat{\mathfrak{e}}\cdot\mathfrak{E}^{\varepsilon}.
\end{array}
\end{equation}
Concerning the first line above in \eqref{nuovebc} we notice that $\hat{\mathfrak{h}}_1= \widehat{\mathcal{H}}_N=0$ on $\omega_T$, so that $\mathfrak{H}^{\varepsilon}_1$ does not appear in the boundary condition.
\\
%%%%%%%%%%%%%%%%%%%%%%%%%%%%%%%%%%%%%%%%%%%%%%%%
%%%%%%%%%%%%%%%%%%%%%%%%%%%%%%%%%%%%%%%%%%%%%%%%
%%%%%%%%%%%%%%%%%%%%%%%%%%%%%%%%%%%%%%%%%%%%%%%%
From \eqref{newmaxwell}, \eqref{nuovebc} we get the new formulation of problem \eqref{34"} for the unknowns $(\mathcal{U}^{\varepsilon}, W^{\varepsilon})$:
%%%
\begin{subequations}\label{77}
\begin{align}
\displaystyle \widehat{\mathcal{A}}_0\partial_t{\mathcal{U}}^{\varepsilon}+\sum_{j=1}^{3}(\widehat{\mathcal{A}}_j+{\mathcal{E}}_{1j+1})\partial_j{\mathcal{U}}^{\varepsilon}+
\widehat{\mathcal C}'{\mathcal{U}}^{\varepsilon}=\mathcal{F} ,
 \qquad &\mbox{in}\ Q^+_T,\label{77a}
\\
{B}_0\partial_tW^{\varepsilon}+\sum_{j=1}^3B_j^{\varepsilon}\partial_jW^{\varepsilon} +{B}_4W^{\varepsilon}=0 \qquad &\mbox{in}\ Q^-_T, \label{78}
\\
\partial_t\varphi^{\varepsilon}+\hat{v}_2\partial_2\varphi^{\varepsilon}+\hat{v}_3\partial_3\varphi^{\varepsilon}-
\varphi^{\varepsilon}\partial_1\hat{v}_{N}-u_1^{\varepsilon} =0, \qquad&\label{79} \\
{q}^{\varepsilon} +  [ \partial_1\hat{q}]\varphi^{\varepsilon}-\hat{\mathfrak{h}}\cdot\mathfrak{H}^{\varepsilon} +\varepsilon\,
\hat{\mathfrak{e}}\cdot\mathfrak{E}^{\varepsilon}=0,  \qquad&\label{80}
\\
\mathfrak{E}^{\varepsilon}_2-\varepsilon\,\partial_t(\widehat{\mathcal{H}}_3\varphi^{\varepsilon} )+\varepsilon\,\partial_2(\widehat{E}_1\varphi^{\varepsilon} )=0, \qquad& \label{81} \\
\mathfrak{E}^{\varepsilon}_3+\varepsilon\,\partial_t(\widehat{\mathcal{H}}_2\varphi^{\varepsilon} )+\varepsilon\,\partial_3(\widehat{E}_1\varphi^{\varepsilon} )=0
 \qquad&\mbox{on}\ \omega_T, \label{82}
\\
(\mathcal{U}^{\varepsilon},W^{\varepsilon},\varphi^{\varepsilon} )=0\quad &\mbox{for}\ t<0.\label{83g}
\end{align}
\end{subequations}
%%%%%%%%%%%%%%%%%%%%%%%%%%%%%%%%%%%%%%%%%%%%%%%%
From \eqref{93"}--\eqref{96"} we get that solutions $(\mathcal{U}^{\varepsilon}, W^{\varepsilon})$ to problem \eqref{77} satisfy
%%%%
\begin{eqnarray}
{\rm div}\,{h}^{\varepsilon}=0\qquad&\mbox{in}\ Q^+_T,
\label{84}
\\
{\rm div}\,{\mathfrak{h}}^{\varepsilon}=0,\quad {\rm div}\,{\mathfrak{e}}^{\varepsilon}=0\qquad&\mbox{in}\ Q^-_T,
\label{85}
\\
{h}_{1}^{\varepsilon}=\widehat{H}_2\partial_2\varphi^{\varepsilon} +\widehat{H}_3\partial_3\varphi^{\varepsilon} -
\varphi^{\varepsilon}\partial_1\widehat{H}_{N},\qquad&
\label{86}
\\
{\mathfrak{h}}_{1}^{\varepsilon} =\partial_2\bigl(\widehat{\mathcal{H}}_2\varphi^{\varepsilon} \bigr) +\partial_3\bigl(\widehat{\mathcal{H}}_3\varphi^{\varepsilon} \bigr)\qquad &\mbox{on}\ \omega_T.
\label{87}
\end{eqnarray}
%%%
%%%%%%%%%%%%%%%%%%%%%%%%%%%%%%%%%%%%%%%%%%%%%%%%
%%%%%%%%%%%%%%%%%%%%%%%%%%%%%%%%%%%%%%%%%%%%%%%%
\begin{remark}\label{}
The invertible part of the boundary matrix of a system allows to control the trace at the boundary of the so-called noncharacteristic component of the vector solution. Thus, with the system \eqref{77a} (whose boundary matrix is $-{\mathcal{E}}_{12}$, because of \eqref{a10}) we have the control of $q^{\varepsilon},u_1^{\varepsilon}$ at the boundary; therefore the components of ${\mathcal{U}}^{\varepsilon}$ appearing in the boundary conditions \eqref{79}, \eqref{80} are well defined.

%The noncharacteristic regularization, with \eqref{34"a} instead of \eqref{34'anew}, allows us to overcome this difficulty, because an invertible boundary matrix gives the control at the boundary of the whole vector solution. The same holds true for \eqref{35"}. We notice however that the control of $\mathfrak{H}_1^{\varepsilon},\mathfrak{E}_1^{\varepsilon}$ given by the boundary matrix of \eqref{35"} is not uniform as $\varepsilon\to0$.
The same holds true for \eqref{78} where we can get the control of $\mathfrak{H}_2^{\varepsilon},\mathfrak{H}_3^{\varepsilon},\mathfrak{E}_2^{\varepsilon},\mathfrak{E}_3^{\varepsilon}$. The control of $\mathfrak{E}_1^{\varepsilon}$ (which appears in \eqref{80}) is not given from the system \eqref{78}, but from the constraint \eqref{85}, as will be shown later on. We recall that $\mathfrak{H}_1^{\varepsilon}$ does not appear in the boundary condition \eqref{80} because $\hat{\mathfrak{h}}_1=\hat{\mathcal{H}}_N=0$.

\end{remark}

%%%%%%%%%%%%%%%%%%%%%%%%%%%%%%%%%%%%%%%%%%%%%%%%
%%%%%%%%%%%%%%%%%%%%%%%%%%%%%%%%%%%%%%%%%%%%%%%%
Before studying problem \eqref{77} (or equivalently \eqref{34"}), we should be sure that the number of boundary conditions is in agreement with the number of incoming characteristics for the hyperbolic systems \eqref{77}. Since one of the four boundary conditions \eqref{79}--\eqref{82} is needed for determining the function $\varphi^{\varepsilon} (t,x')$, the total number of ``incoming'' characteristics should be three. Let us check that this is true.
%%%%%%%%%%%%%%%%%%%%%%
\begin{proposition}\label{nrbc}
If $0<\varepsilon<1$ system \eqref{77a} has one incoming characteristic for the boundary $\omega_T$ of the domain $Q_T^+$.
If $\varepsilon>0$ is sufficiently small, system \eqref{78} has two incoming characteristics for the boundary $\omega_T$ of the domain $Q_T^-$.

\end{proposition}

\begin{proof}
Consider first system \eqref{77a}. In view of  \eqref{a10}, the boundary matrix on $\omega_T$ is $-\mathcal{E}_{12}$ which has one negative (incoming in the domain $Q_T^+$) and one positive eigenvalue, while all other eigenvalues are zero.

Now consider system \eqref{78}.
The boundary matrix $B_1^\varepsilon$ has eigenvalues
$
\lambda_{1,2}=-\varepsilon^{-1},\,
\lambda_{3,4}=\varepsilon^{-1},\,
\lambda_{5,6}=0.
$
%As  $B_0$ is positive definite, the boundary matrix $B_0+B_1^\varepsilon$ as well has two negative eigenvalues for all $\varepsilon>0$ sufficiently small. 
Thus, system \eqref{78} has indeed two incoming characteristics in the domain $Q_T^-$ ($\lambda_{1,2}<0$).
\end{proof}

%\vfil
%\eject

%%%%%%%%%%%%%%%%%%%%%%%%%%%%%%%%%%%%%%%%%%%%
%%%%%%%%%%%%%%%%%%%%%%%%%%%%%%%%%%%%%%%%%%%%
%%%%%%%%%%%%%%%%%%%%%%%%%%%%%%%%%%%%%%%%%%%%
%%%%%%%%%%%%%%%%%%%%%%%%%%%%%%%%%%%%%%%%%%%%
\section{Basic a priori estimate for the hyperbolic regularized problem}\label{basic}

Our goal now is to justify  rigorously the formal limit $\varepsilon \rightarrow 0$ in \eqref{34"}--\eqref{96"}, or alternatively in \eqref{77}--\eqref{87}.
To this end we will prove the existence of solutions to problem \eqref{77}--\eqref{87} and a uniform in $\varepsilon$ a priori estimate. This work will be done in several steps.

\subsection{The boundary value problem}
Assuming that all coefficients and data appearing in \eqref{77} are extended for all times to the whole real line, let us consider the boundary value problem (recall the definition of $Q^\pm,\omega$ in \eqref{defQ})
%%%
\begin{subequations}\label {77'}
\begin{align}
\displaystyle \widehat{\mathcal{A}}_0\partial_t{\mathcal{U}}^{\varepsilon}+\sum_{j=1}^{3}(\widehat{\mathcal{A}}_j+{\mathcal{E}}_{1j+1})\partial_j{\mathcal{U}}^{\varepsilon}+
\widehat{\mathcal C}'{\mathcal{U}}^{\varepsilon}=\mathcal{F} ,
 \qquad &\mbox{in}\ Q^+,\label{77'a}
\\
{B}_0\partial_tW^{\varepsilon}+\sum_{j=1}^3B_j^{\varepsilon}\partial_jW^{\varepsilon} +{B}_4W^{\varepsilon}=0  \qquad &\mbox{in}\ Q^-, \label{78'}
\\
\partial_t\varphi^{\varepsilon}+\hat{v}_2\partial_2\varphi^{\varepsilon}+\hat{v}_3\partial_3\varphi^{\varepsilon}-
\varphi^{\varepsilon}\partial_1\hat{v}_{N}-u_1^{\varepsilon} =0, \qquad&\label{79'} \\
{q}^{\varepsilon} +  [ \partial_1\hat{q}]\varphi^{\varepsilon}-\hat{\mathfrak{h}}\cdot\mathfrak{H}^{\varepsilon} +\varepsilon\,
\hat{\mathfrak{e}}\cdot\mathfrak{E}^{\varepsilon}=0,  \qquad&\label{80'}
\\
\mathfrak{E}^{\varepsilon}_2-\varepsilon\,\partial_t(\widehat{\mathcal{H}}_3\varphi^{\varepsilon} )+\varepsilon\,\partial_2(\widehat{E}_1\varphi^{\varepsilon} )=0, \qquad&\label{81'} \\
\mathfrak{E}^{\varepsilon}_3+\varepsilon\,\partial_t(\widehat{\mathcal{H}}_2\varphi^{\varepsilon} )+\varepsilon\,\partial_3(\widehat{E}_1\varphi^{\varepsilon} )=0
 \qquad&\mbox{on}\ \omega, \label{82'}
\\
(\mathcal{U}^{\varepsilon},W^{\varepsilon},\varphi^{\varepsilon} )=0\qquad &\mbox{for}\ t<0.\label{83g'}
\end{align}
\end{subequations}
%%%%%%%%%%%%%%%%%%%%%%%%%%%%%%%%%%%%%%%%%%%%%%%%
%%%%%%%%%%%%%%%%%%%%%%%%%%%%%%%%%%%%%%%%%%%%%%%%
%%%%%%%%%%%%%%%%%%%%%%%%%%%%%%%%%%%%%%%%%%%%%%%%
%%%%%%%%%%%%%%%%%%%%%%%%%%%%%%%%%%%%%%%%%%%%%%%%
%%%%%%%%%%%%%%%%%%%%%%%%%%%%%%%%%%%%%%%%%%%%%%%%
%%%%%%%%%%%%%%%%%%%%%%%%%%%%%%%%%%%%%%%%%%%%%%%%
In this section we prove a uniform in $\varepsilon$ a priori estimate of smooth solutions of \eqref{77'}.

\begin{theorem}\label{lem1}
Let the basic state \eqref{21} satisfies assumptions \eqref{22}--\eqref{27} and \eqref{41} for all times. There exist ${\varepsilon}_0>0,\,\gamma_0\ge1$ such that if  $0<{\varepsilon}<{\varepsilon}_0$ and $\gamma\ge\gamma_0$ then all sufficiently smooth solutions $(\mathcal{U}^{\varepsilon},W^{\varepsilon},\varphi^{\varepsilon})$ of problem \eqref{77'} obey the estimate
\begin{multline}\label{54'}
\gamma \left(\|\mathcal{U}^{\varepsilon}_\gamma\|^2_{H^{1}_{tan,\gamma}(Q^+)}+\|W^{\varepsilon}_\gamma\|^2_{H^{1}_\gamma(Q^-)}
+\|(q^{\varepsilon}_\gamma,u_{1\gamma}^{\varepsilon},h_{1\gamma}^{\varepsilon})|_{\omega}\|^2_{H^{1/2}_\gamma(\omega)}
+
\|W^{\varepsilon}_{\gamma}|_{\omega}\|^2_{H^{1/2}_\gamma(\omega)}\right)
\\
+\gamma^2\|\varphi^{\varepsilon}_\gamma\|^2_{H^1_\gamma(\omega)}
\leq
\frac{C}{\gamma}\|\mathcal{F}_\gamma\|^2_{H^{1}_{tan,\gamma}(Q^+)} ,\end{multline}
where we have set $\mathcal{U}^{\varepsilon}_\gamma=e^{-\gamma t}\,\mathcal{U}^{\varepsilon}, W^{\varepsilon}_\gamma=e^{-\gamma t}\, W^{\varepsilon}, \varphi^{\varepsilon}_\gamma= e^{-\gamma t}\, \varphi^{\varepsilon}$
and so on, and where $C=C(K,\delta)>0$ is a constant independent of the data $\mathcal{F}$ and the parameters $\varepsilon,\gamma$.
\end{theorem}

Passing to the limit $\varepsilon \to 0$ in this estimate will give the a priori estimate \eqref{54}.

%%%%%
Since problem \eqref{77'} looks similar to a corresponding one in relativistic MHD \cite{trakhinin10a},
for the deduction of estimate \eqref{54'} we use the same ideas as in \cite{trakhinin10a}. On the one hand, we even have an advantage, in comparison with the problem in \cite{trakhinin10a}, because the coefficients $\widehat{E}_j$ in \eqref{78'}, \eqref{80'}--\eqref{82'} are still arbitrary functions whose choice will be crucial to make boundary conditions dissipative. On the other hand, we should be more careful with lower-order terms than in \cite{trakhinin10a}, because we must avoid the appearance of terms with $\varepsilon^{-1}$ (otherwise, our estimate will not be uniform in $\varepsilon$). Also for this reason we are using the variables $(\mathcal{U}^{\varepsilon}, W^{\varepsilon})$ rather than $(U^{\varepsilon},V^{\varepsilon})$.

For the proof of \eqref{54'} we will need a secondary symmetrization of the transformed Maxwell equations in vacuum \eqref{35"}, \eqref{94"}.

%%%%%%%%%%%%%%%%%%%%%%%%%%%%%%%%%%%%%%%%%%%%%%%
%%%%%%%%%%%%%%%%%%%%%%%%%%%%%%%%%%%%%%%%%%%%%%%
%%%%%%%%%%%%%%%%%%%%%%%%%%%%%%%%%%%%%%%%%%%%%%%
\subsection{ A secondary symmetrization}
In order to show how to get the secondary symmetrization, for the sake of simplicity we consider first a planar unperturbed interface, i.e. the case $\hat{\varphi}\equiv 0$. For this case \eqref{35"}, \eqref{94"} become
\begin{equation}
\partial_t V^{\varepsilon} +\sum_{j=1}^3B_k^{\varepsilon}\partial_kV^{\varepsilon}=0,
\label{13'}
\end{equation}
\begin{equation}
{\rm div}\,\mathcal{H}^{\varepsilon}=0,\quad {\rm div}\,E^{\varepsilon}=0.
\label{14}
\end{equation}
We write for system \eqref{13'} the following secondary symmetrization (for a similar secondary symmetrization of the Maxwell equations in vacuum see \cite{trakhinin10a}):
\begin{equation}
\mathfrak{B}_0^{\varepsilon}\partial_tV^{\varepsilon}  +\sum_{j=1}^3\mathfrak{B}_0^{\varepsilon}B_j^{\varepsilon}\partial_jV^{\varepsilon} +R_1{\rm div}\,\mathcal{H}^{\varepsilon}
+R_2{\rm div}\, E^{\varepsilon}
 =\mathfrak{B}_0^{\varepsilon}\partial_tV^{\varepsilon} +\sum_{j=1}^3\mathfrak{B}_j^{\varepsilon}\partial_jV^{\varepsilon}=0,
\label{20'}
\end{equation}
where
\begin{equation}
\begin{array}{ll}\label{B0e}
\mathfrak{B}_0^{\varepsilon}=\left(\begin{array}{cccccc}
1 & 0 & 0& 0 & {\varepsilon}\nu_3 & -{\varepsilon}\nu_2 \\
0 & 1 & 0& -{\varepsilon}\nu_3 & 0 & {\varepsilon}\nu_1 \\
0 & 0 & 1& {\varepsilon}\nu_2 & -{\varepsilon}\nu_1 & 0 \\
0 & -{\varepsilon}\nu_3 & {\varepsilon}\nu_2& 1 & 0 & 0 \\
{\varepsilon}\nu_3 & 0 & -{\varepsilon}\nu_1& 0 & 1 & 0 \\
-{\varepsilon}\nu_2 & {\varepsilon}\nu_1 & 0& 0 & 0 & 1
\end{array} \right),
\end{array}
\end{equation}
\[
\mathfrak{B}_1^{\varepsilon}=
\left(\begin{array}{cccccc}
\nu_1 & \nu_2 & \nu_3& 0 & 0 & 0 \\
\nu_2 & -\nu_1 & 0& 0 & 0 & -\varepsilon^{-1} \\
\nu_3 & 0 & -\nu_1& 0 & \varepsilon^{-1} & 0 \\
0 & 0 & 0& \nu_1 & \nu_2 & \nu_3 \\
0 & 0 & \varepsilon^{-1}& \nu_2 & -\nu_1 & 0 \\
0 & -\varepsilon^{-1} & 0& \nu_3 & 0 & -\nu_1
\end{array} \right),
\quad
\mathfrak{B}_2^{\varepsilon}=
\left(\begin{array}{cccccc}
-\nu_2 & \nu_1 & 0& 0 & 0 & \varepsilon^{-1} \\
\nu_1 & \nu_2 & \nu_3& 0 & 0 & 0 \\
0 & \nu_3 & -\nu_2& -\varepsilon^{-1} & 0 & 0 \\
0 & 0 & -\varepsilon^{-1}& -\nu_2 & \nu_1 & 0 \\
0 & 0 & 0& \nu_1 & \nu_2 & \nu_3 \\
\varepsilon^{-1} & 0 & 0& 0 & \nu_3 & -\nu_2
\end{array} \right),
\]
\[
\mathfrak{B}_3^{\varepsilon}=
\left(\begin{array}{cccccc}
-\nu_3 & 0 & \nu_1& 0 & -\varepsilon^{-1} & 0 \\
0 & -\nu_3 & \nu_2& \varepsilon^{-1} & 0 & 0 \\
\nu_1 & \nu_2 & \nu_3& 0 & 0 & 0 \\
0 & \varepsilon^{-1} & 0& -\nu_3 & 0 & \nu_1 \\
-\varepsilon^{-1} & 0 & 0& 0 & -\nu_3 & \nu_2 \\
0 & 0 & 0& \nu_1 & \nu_2 & \nu_3
\end{array} \right),
\quad
R_1=\left(\begin{array}{c} \nu_1 \\
\nu_2 \\
\nu_3 \\
0 \\
0 \\
0
\end{array} \right),\quad R_2=\left(\begin{array}{c}
0 \\
0 \\
0 \\
\nu_1 \\
\nu_2 \\
\nu_3
\end{array} \right).
\]
The arbitrary functions $\nu_i(t,x)$ will be chosen in appropriate way later on.
 It may be useful to notice that system \eqref{20'} can also be written as
 \begin{equation}
\begin{array}{ll}\label{secsym}
\displaystyle
(\partial_t \mathcal{H}^{\varepsilon}+\frac{1}{\varepsilon}\nabla\times E^{\varepsilon}) - \vec\nu\times
(\varepsilon\partial_t E^{\varepsilon}-\nabla\times \mathcal{H}^{\varepsilon})+
\vec\nu\,{\rm div}\,\mathcal{H}^{\varepsilon}=0,
\\
\\
\displaystyle
(\partial_t E^{\varepsilon}-\frac{1}{\varepsilon}\nabla\times \mathcal{H}^{\varepsilon}) + \vec\nu\times
(\varepsilon\partial_t \mathcal{H}^{\varepsilon}+\nabla\times E^{\varepsilon})+
\vec\nu\, {\rm div}\,E^{\varepsilon}=0,
\end{array}
\end{equation}
with the vector-function $\vec\nu = (\nu_1, \nu_2, \nu_3)$. The symmetric system \eqref{20'} (or \eqref{secsym}) is hyperbolic if $\mathfrak{B}_0^{\varepsilon}>0$, i.e. for
\begin{equation}
\begin{array}{ll}\label{99}
{\varepsilon}|\vec\nu |<1.
\end{array}
\end{equation}
The last inequality is satisfied for any given $\nu$ and small ${\varepsilon}$. We compute\footnote{The manual computation of the determinants is definitely too long. Here we used a free program for symbolic calculus, with the help of PS's son Martino.}
$$
\mbox{det}(\mathfrak{B}_1^{\varepsilon})=\nu_1^2\left(|\vec\nu|^2-1/\epsilon^2\right)^2.
$$
Therefore the boundary is noncharacteristic for system \eqref{20'} (or \eqref{secsym}) provided \eqref{99} and $\nu_1\not=0$ hold.

%%%%%%%%%%%%%

Consider now a nonplanar unperturbed interface, i.e., the general case when $\hat{\varphi}$ is not identically zero.
Similarly to \eqref{20'}, from \eqref{newmaxwell}, \eqref{94"} we get the secondary symmetrization
\[
K\mathfrak{B}_0^{\varepsilon}K^{-1}\left({B}_0\partial_tW^{\varepsilon}+\sum_{j=1}^3B_j^{\varepsilon}\partial_jW^{\varepsilon} +{B}_4W^{\varepsilon}\right)+
\frac{1}{\partial_1\widehat{\Phi}_1}K\Big(R_1{\rm div}\,\mathfrak{h}^{\varepsilon}
+R_2{\rm div}\,\mathfrak{e}^{\varepsilon}\Big)=0.
\]
%%%%%%%%%%%%%%%%%%%%%%%%%%%%%%%%%%%%
\\
We write this system as
\begin{equation}
M_0^{\varepsilon}\partial_tW^{\varepsilon} +\sum_{j=1}^3M_j^{\varepsilon}\partial_jW^{\varepsilon}+M_4^{\varepsilon}W^{\varepsilon}=0,
\label{50"}
\end{equation}
where
\begin{equation}
\begin{array}{ll}\label{defmatrices}
\ds M_0^{\varepsilon}=\frac{1}{\partial_1\widehat{\Phi}_1}\,K\mathfrak{B}_0^{\varepsilon}K^{\textsf{T}}>0,
\quad
\ds M_j^{\varepsilon}=\frac{1}{\partial_1\widehat{\Phi}_1}\,K\mathfrak{B}_j^{\varepsilon}K^{\textsf{T}} \quad (j=2,3),
\\
\ds M_1^{\varepsilon}=\frac{1}{\partial_1\widehat{\Phi}_1}\,K
\widetilde{\mathfrak{B}}_1^{\varepsilon}K^{\textsf{T}}, \quad
\widetilde{\mathfrak{B}}_1^{\varepsilon}=\frac{1}{\partial_1\widehat{\Phi}_1}\Bigl(
\mathfrak{B}_1^{\varepsilon}-\sum_{k=2}^3\mathfrak{B}_k^{\varepsilon}\partial_k\widehat{\Psi} \Bigr),
\\
\ds {M}_4^{\varepsilon}=
K\left({\mathfrak{B}}_0^{\varepsilon}\partial_t+\widetilde{\mathfrak{B}}_1^{\varepsilon}\partial_1+\mathfrak{B}_2^{\varepsilon} \partial_2+\mathfrak{B}_3^{\varepsilon} \partial_3+ \mathfrak{B}_0^{\varepsilon}{B}_4\right)
\left( \frac{1}{\partial_1\widehat{\Phi}_1}\,K^T \right).

\end{array}
\end{equation}
%%%
%%%
System \eqref{50"} is symmetric hyperbolic provided that \eqref{99} holds. We compute
\begin{equation}
\begin{array}{ll}\label{determinant}
\mbox{det}({M}_1^{\varepsilon})=\left(1+(\ddue\hat\varphi)^2+(\dtre\hat\varphi)^2\right)^2\left(\nu_1-\nu_2\ddue\hat\varphi -\nu_3\dtre\hat\varphi \right)^2\left(|\vec\nu|^2-1/\epsilon^2\right)^2 ,
\end{array}
\end{equation}
%\begin{equation}
%\begin{array}{ll}\label{determinant2}
%\mbox{det}({M}_1^{\varepsilon})=\left(1+(\ddue\hat\varphi)^2+(\dtre\hat\varphi)^2\right)^2\left(\nu_1-\nu_2\ddue\hat\varphi -\nu_3\dtre\hat\varphi \right)^2\left(|\nu|^2-1/\epsilon^2\right)^2 \quad\mbox{ on }\omega,
%\end{array}
%\end{equation}
and so the boundary is noncharacteristic for system \eqref{50"} if and only if \eqref{99} holds and $\nu_1\not=\nu_2\ddue\hat\varphi +\nu_3\dtre\hat\varphi $.
System \eqref{50"} originates from a linear combination of equations \eqref{35"} similar to \eqref{secsym}, namely
from
\begin{equation}
\begin{array}{ll}\label{secsym2}
\displaystyle
(\partial_t \mathfrak{h}^{\varepsilon}+\frac{1}{\varepsilon}\nabla\times \mathfrak{E}^{\varepsilon}) -
\hat\eta\left(\vec\nu\times\hat\eta^{-1}(\varepsilon\partial_t \mathfrak{e}^{\varepsilon}-\nabla\times \mathfrak{H}^{\varepsilon})\right)+
\frac{\hat\eta\,\vec\nu}{\partial_1\widehat{\Phi}_1}\,{\rm div}\,\mathfrak{h}^{\varepsilon}=0,
\\
\\
\displaystyle
(\partial_t \mathfrak{e}^{\varepsilon}-\frac{1}{\varepsilon}\nabla\times \mathfrak{H}^{\varepsilon}) +
\hat\eta\left(\vec\nu\times\hat\eta^{-1}(\varepsilon\partial_t \mathfrak{h}^{\varepsilon}+\nabla\times \mathfrak{E}^{\varepsilon})\right)+
\frac{\hat\eta\,\vec\nu}{\partial_1\widehat{\Phi}_1}\, {\rm div}\,\mathfrak{e}^{\varepsilon}=0.
\end{array}
\end{equation}
%%%%%%%%%%%%%%%%%%%%%%%%%%%%%%%%%%%%%%
We need to know which is the behavior of the above matrices in \eqref{defmatrices} w.r.t. $\varepsilon$ as $\varepsilon \rightarrow 0$.
In view of this, let us denote a generic matrix which is bounded w.r.t. $\varepsilon$ by $O(1)$.
Looking at \eqref{secsym2} we immediately find
\begin{equation}
\begin{array}{ll}\label{ordine}
\ds M_0^{\varepsilon}=O(1),
\quad
M_j^{\varepsilon}={B}_j^{\varepsilon}+O(1)
\quad (j=1, 2,3),
\quad M_4^{\varepsilon}=O(1).
\end{array}
\end{equation}
As the matrices $M_0^{\varepsilon}$ and ${M}_4^{\varepsilon}$ do not contain the multiplier $\varepsilon^{-1}$, their norms are bounded as $\varepsilon \rightarrow 0$.
Recalling that the matrices $B_j^{\varepsilon}$ are constant, we deduce as well that all the possible derivatives (with respect to $t$ and $x_j$) of the matrices $M_j^{\varepsilon}$ have bounded norms as $\varepsilon \rightarrow 0$.

%%%%%%%%%%%%%%%%%%%%%%%%%%%%%%%%%%%%%%%%%%%%%
%%%%%%%%%%%%%%%%%%%%%%%%%%%%%%%%%%%%%%%%%%%%%
\subsection{Proof of Theorem \ref{lem1}}

For the proof of our basic a priori estimate \eqref{54'} we  will apply the energy method to the symmetric hyperbolic systems \eqref{77'a} and \eqref{50"}.
In the sequel $\gamma_0\ge1$ denotes a generic constant sufficiently large which may increase from formula to formula, and $C$ is a generic constant that may change from line to line.

First of all we provide some preparatory estimates.
In particular, to estimate the weighted conormal derivative  $Z_1=\sigma\partial_1$ of $\mathcal{U}^{\varepsilon}$ (recall the definition \eqref{normaconormale} of the $\gamma$-dependent norm of $H^1_{tan,\gamma}$) we do not need any boundary condition because the weight $\sigma$ vanishes on $\omega$. Applying to system \eqref{77'a} the operator $Z_1$ and using standard arguments of the energy method\footnote{
We multiply $Z_1$\eqref{77'a} by $e^{-\gamma t}\,Z_1\mathcal{U}^{\varepsilon}_\gamma$ and integrate by parts over $Q^+$, then we use the Cauchy-Schwarz inequality.
}, yields the inequality
\begin{equation}
\gamma \|Z_1\mathcal{U}_\gamma^{\varepsilon}\|^2_{L^2(Q^+)} \leq \frac{C}{\gamma}\left\{ \|\mathcal{F}_\gamma\|^2_{H^1_{tan,\gamma}(Q^+)}+\|\mathcal{U}_\gamma^{\varepsilon}\|^2_{H^1_{tan,\gamma}(Q^+)}+\|\mathcal{E}_{12}\mathcal{U}_\gamma^{\varepsilon}\|^2_{L^2(Q^+)} \right\}
,
\label{126}
\end{equation}
for $\gamma\ge \gamma_0$. On the other hand, directly from the equation \eqref{77'a} we have
\begin{equation}
\begin{array}{ll}\label{126'}
\|\mathcal{E}_{12}\mathcal{U}_\gamma^{\varepsilon}\|^2_{L^2(Q^+)}\le C
\left\{ \|\mathcal{F}_\gamma\|^2_{L^2(Q^+)}+\|\mathcal{U}_\gamma^{\varepsilon}\|^2_{H^1_{tan,\gamma}(Q^+)} \right\}
,
\end{array}
\end{equation}
where $C$ is independent of ${\varepsilon},\gamma$.
Thus from \eqref{126}, \eqref{126'} we get
\begin{equation}
\gamma \|Z_1\mathcal{U}_\gamma^{\varepsilon}\|^2_{L^2(Q^+)} \leq \frac{C}{\gamma}\left\{ \|\mathcal{F}_\gamma\|^2_{H^1_{tan,\gamma}(Q^+)}+\|\mathcal{U}_\gamma^{\varepsilon}\|^2_{H^1_{tan,\gamma}(Q^+)} \right\}
, \quad \gamma\ge \gamma_0,
\label{126''}
\end{equation}
where $C$ is independent of ${\varepsilon},\gamma$.
Furthermore, using the special structure of the boundary matrix in \eqref{77'a} (see \eqref{a10}) and the divergence constraint
\eqref{84}, we may estimate the normal derivative of the noncharacteristic part $\mathcal{U}_{n\gamma}^{\varepsilon}=e^{-\gamma t}(q^{\varepsilon},u_1^{\varepsilon},h_1^{\varepsilon})$ of the ``plasma'' unknown $\mathcal{U}_{\gamma}^{\varepsilon}$:
\begin{equation}
\|\partial_1 \mathcal{U}_{n\gamma}^{\varepsilon} \|^2_{L^2(Q^+)} \leq  C\left\{ \|\mathcal{F}_\gamma\|^2_{L^2(Q^+)} +\|\mathcal{U}_\gamma^{\varepsilon}\|^2_{H^1_{tan,\gamma}(Q^+)} \right\},
\label{130}
\end{equation}
where $C$ is independent of ${\varepsilon},\gamma$.
In a similar way we wish to express the normal derivative of $W^{\varepsilon}$ through its tangential derivatives. Here it is convenient to use system \eqref{78'} rather than \eqref{50"}. We multiply \eqref{78'} by $\varepsilon$ and find from the obtained equation an explicit expression for the normal derivatives of $\mathfrak{H}_2^{\varepsilon},\mathfrak{H}_3^{\varepsilon},\mathfrak{E}^{\varepsilon}_2,\mathfrak{E}_3^{\varepsilon}$. An explicit expression for the normal derivatives of $\mathfrak{H}^{\varepsilon}_1,\mathfrak{E}_1^{\varepsilon}$ is found through the divergence constraints \eqref{85}.
Thus we can estimate the normal derivatives of all the components of $W^{\varepsilon}$ through its tangential derivatives:
\begin{equation}
\|\partial_1W_\gamma^{\varepsilon}\|^2_{L_2(Q^-)} \leq C
\left\{\gamma^2 \|W_\gamma^{\varepsilon} \|^2_{L_2(Q^-)}+\|\partial_tW_\gamma^{\varepsilon} \|^2_{L_2(Q^-)}+
\sum_{k=2}^3\|\partial_kW_\gamma^{\varepsilon} \|^2_{L_2(Q^-)} \right\}
,\label{127}
\end{equation}
where $C$ does not depend on $\varepsilon$  and $\gamma$, for all ${\varepsilon}\le{\varepsilon}_0$.

%%%%%%%%%%%%%%%%%%%%%%%%%%%%%%%%%%%%%%%%%%%%%%%%%%%%
%%%%%%%%%%%%%%%%%%%%%%%%%%%%%%%%%%%%%%%
%%%%%%%%%%%%%%%%%%%%%%%%%%%%%%%%%%%%%%%
As for the front function $\varphi^\epsilon$ we easily obtain from \eqref{79'} the $L^2$ estimate
\begin{equation}
\gamma\|\varphi_{\gamma}^{\varepsilon}\|^2_{L^2(\omega)}
\\
 \leq \frac{C}{\gamma}
\|u_{1\gamma}^{\varepsilon}\|^2_{L^2(\omega)}, \quad   \gamma\ge \gamma_0,
\label{phi}
\end{equation}
where $C$ is independent of $\gamma$.
Furthermore, thanks to our basic assumption \eqref{41}\footnote{Under the conditions $\hat{H}_N=\hat{\mathcal{H}}_N=0$ one has
$|\hat{H}\times\hat{\mathcal{H}}|^2=(\hat{H}_2\hat{\mathcal{H}}_3-\hat{H}_3\hat{\mathcal{H}}_2)^2\langle\nabla'\hat\varphi\rangle^2$ on $\omega$,
where we have set $\langle\nabla'\hat\varphi\rangle:=(1+|\ddue\hat\varphi|^2+|\dtre\hat\varphi|^2)^{1/2}$.}
we can resolve \eqref{86}, \eqref{87} and \eqref{79'} for the space-time gradient $\nabla_{t,x'}\varphi_\gamma^{\varepsilon} =(\partial_t\varphi^{\varepsilon}_\gamma,\partial_2\varphi^{\varepsilon}_\gamma,\partial_3\varphi^{\varepsilon}_\gamma)$:
\begin{equation}
\nabla_{t,x'}\varphi^{\varepsilon}_{\gamma}=\hat{a}_1{h}^{\varepsilon}_{1\gamma}+\hat{a}_2{\mathfrak{h}}^{\varepsilon}_{1\gamma} +\hat{a}_3{u}^{\varepsilon}_{1\gamma}+\hat{a}_4\varphi^{\varepsilon}_\gamma+\gamma\hat{a}_5\varphi^{\varepsilon}_\gamma ,
\label{12a}
\end{equation}
where the vector-functions $\hat{a}_{\alpha}={a}_{\alpha}(\widehat{U}_{|\omega},\widehat{\mathcal{H}}_{|\omega})$ of coefficients can be easily written in explicit form. From \eqref{12a} we get
\begin{equation}
\begin{array}{ll}\label{gradphi}
\|\nabla_{t,x'}\varphi^{\varepsilon}_{\gamma}\|_{L^2(\omega)}\leq
C\left( \|\mathcal{U}_{n\gamma}^{\varepsilon}|_\omega\|_{L^2(\omega)}+\|W_{\gamma}^{\varepsilon}|_\omega \|_{L^2(\omega)}+ \gamma\|\varphi^{\varepsilon}_\gamma\|_{L^2(\omega)}
\right).
\end{array}
\end{equation}

%%%%%%%%%%%%%%%%%%%%%%%%%%%%%%%%%%%%%%%%%%%%%%%%%%%%
%%%%%%%%%%%%%%%%%%%%%%%%%%%%%%%%%%%%%%%%%%%%%%%%%%%%

Now we prove a $L^2$ energy estimate for $(\mathcal{U}^{\varepsilon},W^{\varepsilon})$. We multiply \eqref{77'a} by $e^{-\gamma t}\,\mathcal{U}^{\varepsilon}_\gamma$ and \eqref{50"} by $e^{-\gamma t}\,W^{\varepsilon}_\gamma$, integrate by parts over $Q^\pm$, then we use the Cauchy-Schwarz inequality.
We easily obtain
\begin{multline}
\gamma\int_{Q^+}(\widehat{A}_0{\mathcal{U}}^{\varepsilon}_\gamma,{\mathcal{U}}^{\varepsilon}_\gamma)\,dxdt+
\gamma\int_{Q^-}(M_0^{\varepsilon}W^{\varepsilon}_\gamma,W_\gamma^{\varepsilon})dxdt
+ \int_{\omega}\mathcal{A}^{\varepsilon}\,dx'dt 
\\
 \leq C
\left\{ \frac{1}{\gamma}\|\mathcal{F}_\gamma\|^2_{L^2(Q^+)}+\|\mathcal{U}_\gamma^{\varepsilon}\|^2_{L^2(Q^+)} +\|W_\gamma^{\varepsilon}\|^2_{L^2(Q^-)} \right\},
\label{19a0}
\end{multline}
where we have denoted
\[
\mathcal{A}^{\varepsilon}=-\frac{1}{2}({\mathcal{E}}_{12}{\mathcal{U}}^{\varepsilon}_{\gamma},{\mathcal{U}}^{\varepsilon}_{\gamma})|_{\omega}
+\frac{1}{2}(M_1^{\varepsilon}W^{\varepsilon}_{\gamma},W^{\varepsilon}_{\gamma})|_{\omega}.
\]
%We recall that ${M}_0^{\varepsilon}$ is definite positive, ${M}_0^{\varepsilon}\ge{m}_0>0$, with a constant $m_0$ independent of ${\varepsilon}$, for all ${\varepsilon}$ sufficiently small (see \eqref{B0e}, \eqref{defmatrices}), say $0<\varepsilon<\varepsilon_0$. 
Thanks to the properties of the matrices $M_{\alpha}^{\varepsilon}$ ($\alpha =\overline{0,4}$) described in \eqref{ordine}, the constant $C$ in \eqref{19a0} is uniformly bounded in ${\varepsilon}$ and $\gamma$.
%
%%%%%%%%%%%%%%%%%%%%%%%%%%%%%%%%%%%%%%%%%%%%%%%%%%%%
Let us calculate the quadratic form $\mathcal{A}^{\varepsilon}$
for the following choice of the functions $\nu_j$ in the secondary symmetrization\footnote{Notice that the choice \eqref{nu} makes the boundary characteristic, see \eqref{determinant}.}:
\begin{equation}
\nu_1=\hat{v}_2\partial_2\hat{\varphi}+\hat{v}_3\partial_3\hat{\varphi},\quad \nu_k=\hat{v}_k,\quad k=2,3.
\label{nu}
\end{equation}
After long calculations we get (for simplicity we drop the index $\gamma$)
\begin{equation}
\begin{array}{ll}

\mathcal{A}^{\varepsilon}
=-q^{\varepsilon}u_1^{\varepsilon}+\varepsilon^{-1}({\mathfrak{H}}_3^{\varepsilon}{\mathfrak{E}}_2^{\varepsilon}-{\mathfrak{H}}_2^{\varepsilon}{\mathfrak{E}}_3^{\varepsilon})
+(\hat v_2 {\mathfrak{H}}_2^{\varepsilon}+\hat v_3 {\mathfrak{H}}_3^{\varepsilon}){\mathcal{H}}_N^{\varepsilon}
+(\hat v_2 {\mathfrak{E}}_2^{\varepsilon}+\hat v_3 {\mathfrak{E}}_3^{\varepsilon}){E}_N^{\varepsilon},
\label{A}
\quad \mbox{on  }\omega.
\end{array}
\end{equation}
Now we insert the boundary conditions \eqref{87}, \eqref{79'}--\eqref{82'} in the quadratic form $\mathcal{A}^{\varepsilon}$, recalling also $\hat{\mathcal{H}}_{N}|_{\omega}=0$ and noticing that
$$ \hat{\mathfrak{e}}\cdot\mathfrak{E}^{\varepsilon}
= \widehat{E}_1{E}^{\varepsilon}_N+\widehat{E}_{\tau_2}{E}^{\varepsilon}_2+\widehat{E}_{\tau_3}{E}^{\varepsilon}_3
=\hat{\mathfrak{E}}\cdot\mathfrak{e}^{\varepsilon}.
$$
Again after long calculations we get
%%%%%%%
%YURI TRAKHININ:THE FIRST LINE BECOMES
%$$
%\mathcal{A}^{\varepsilon}=\bigl(\widehat{E}_1+\hat{v}_2\widehat{\mathcal{H}}_3-\hat{v}_3
%\widehat{\mathcal{H}}_2\bigr)\bigl( \varepsilon{E}_N^{\varepsilon}\partial_t\varphi^{\varepsilon}+({\mathfrak{H}}^{\varepsilon}_{2}-\varepsilon E_3)\partial_3\varphi^{\varepsilon}
%-({\mathfrak{H}}^{\varepsilon}_{3}+\varepsilon E_2))\partial_2\varphi^{\varepsilon}\bigr)
%$$
%%%%%%%%%
\begin{multline}\label{AA}
\mathcal{A}^{\varepsilon}=\bigl(\widehat{E}_1+\hat{v}_2\widehat{\mathcal{H}}_3-\hat{v}_3
\widehat{\mathcal{H}}_2\bigr)\bigl( \varepsilon{E}_N^{\varepsilon}\partial_t\varphi^{\varepsilon}+{\mathfrak{H}}^{\varepsilon}_{2}\partial_3\varphi^{\varepsilon}
-{\mathfrak{H}}^{\varepsilon}_{3}\partial_2\varphi^{\varepsilon}\bigr)
\\
+\bigl(\varepsilon \widehat{E}_{\tau_2}E_2^{\varepsilon}+\varepsilon \widehat{E}_{\tau_3}E_3^{\varepsilon} \bigr)\bigl(\partial_t\varphi^{\varepsilon} +\hat{v}_2\partial_2\varphi^{\varepsilon} +\hat{v}_3\partial_3\varphi^{\varepsilon} \bigr)
\\
+\varphi^{\varepsilon} \big\{-\gamma q^{\varepsilon}+ [ \partial_1\hat{q}]\,u_1^{\varepsilon}- \partial_1\hat{v}_N(q^{\varepsilon}+[ \partial_1\hat{q}]\varphi^{\varepsilon})
+ (\gamma \hat{\mathcal{H}}_3+\partial_t\hat{\mathcal{H}}_3 - \partial_2\hat{E}_1)(\mathfrak{H}^{\varepsilon}_{3} + \varepsilon\hat{v}_2 E_N^{\varepsilon})
\\
+ (\gamma \hat{\mathcal{H}}_2+\partial_t\hat{\mathcal{H}}_2 + \partial_3\hat{E}_1)(\mathfrak{H}^{\varepsilon}_{2} - \varepsilon\hat{v}_3 E_N^{\varepsilon})
+(\partial_2\hat{\mathcal{H}}_2 + \partial_3\hat{\mathcal{H}}_3 ) (\hat{v}_2\mathfrak{H}_{2}^{\varepsilon} + \hat{v}_3\mathfrak{H}_{3}^{\varepsilon}   )\big\}
\quad \mbox{on  }{\omega}
\, .
\end{multline}
%%%%%%
Thanks to the multiplicative factor $\varepsilon$ in the boundary condition \eqref{81'}, \eqref{82'}, the critical term with the multiplier ${\varepsilon}^{-1}$ in \eqref{A} has been dropped out.
We make the following choice of the coefficients $\widehat{E}_j$ in the boundary conditions \eqref{80'}--\eqref{82'}:
\begin{equation*}
\widehat{E}= \widehat{\mathcal{H}}\times\vec\nu,
\label{}
\end{equation*}
where $\vec\nu$ is that of \eqref{nu}. For this choice
\begin{equation}\label{E}
\widehat{E}_1+\hat{v}_2\widehat{\mathcal{H}}_3-\hat{v}_3
\widehat{\mathcal{H}}_2=0,\quad \widehat{E}_{\tau_2}=0,\quad \widehat{E}_{\tau_3}=0,
\end{equation}
and this leaves us with
\begin{multline*}\label{}
\mathcal{A}^{\varepsilon}=
\varphi^{\varepsilon} \big\{ -\gamma (q^\varepsilon + [ \partial_1\hat{q}]\varphi^{\varepsilon})+ [ \partial_1\hat{q}]\,(u_1^{\varepsilon}+\varphi^{\varepsilon}\partial_1\hat{v}_N) + \partial_1\hat{v}_Nq^{\varepsilon}
\\
+ ( \gamma \hat{\mathcal{H}}_3 + \partial_t\hat{\mathcal{H}}_3 - \partial_2\hat{E}_1)(\mathfrak{H}^{\varepsilon}_{3} + \varepsilon\hat{v}_2 E_N^{\varepsilon})
+ ( \gamma \hat{\mathcal{H}}_2 + \partial_t\hat{\mathcal{H}}_2 + \partial_3\hat{E}_1)(\mathfrak{H}^{\varepsilon}_{2} - \varepsilon\hat{v}_3 E_N^{\varepsilon})
\\
+(\partial_2\hat{\mathcal{H}}_2 + \partial_3\hat{\mathcal{H}}_3 ) (\hat{v}_2\mathfrak{H}_{2}^{\varepsilon} + \hat{v}_3\mathfrak{H}_{3}^{\varepsilon}   )\big\}
\quad \mbox{on  }{\omega}
\, .
\end{multline*}
Then we write in more convenient form the terms with coefficient $\gamma$ substituting from \eqref{80'}
$$
-({q}^{\varepsilon}+ [ \partial_1\hat{q}]\varphi^{\varepsilon} ) +\hat{\mathcal{H}}_2\mathfrak{H}^{\varepsilon}_2+\hat{\mathcal{H}}_3\mathfrak{H}^{\varepsilon}_3
=  \varepsilon\,
\hat{\mathfrak{e}}\cdot\mathfrak{E}^{\varepsilon},
$$
and we notice that
$$
\hat{\mathfrak{e}}\cdot\mathfrak{E}^{\varepsilon}+(\hat{v}_2\hat{\mathcal{H}}_3 - \hat{v}_3\hat{\mathcal{H}}_2     )E_N^{\varepsilon}=\hat{\mathfrak{E}}\cdot\mathfrak{e}^{\varepsilon}+(\hat{v}_2\hat{\mathcal{H}}_3 - \hat{v}_3\hat{\mathcal{H}}_2     )\mathfrak{e}^{\varepsilon}_1=0, \quad \mbox{on  }{\omega},
$$
again by \eqref{E}. Thus we get
\begin{multline}\label{AAA}
\mathcal{A}^{\varepsilon}=
\varphi^{\varepsilon} \big\{  [ \partial_1\hat{q}]\,(u_1^{\varepsilon}+\varphi^{\varepsilon}\partial_1\hat{v}_N) + \partial_1\hat{v}_Nq^{\varepsilon}
\\
+ (  \partial_t\hat{\mathcal{H}}_3 - \partial_2\hat{E}_1)(\mathfrak{H}^{\varepsilon}_{3} + \varepsilon\hat{v}_2 E_N^{\varepsilon})
+ ( \partial_t\hat{\mathcal{H}}_2 + \partial_3\hat{E}_1)(\mathfrak{H}^{\varepsilon}_{2} - \varepsilon\hat{v}_3 E_N^{\varepsilon})\\
+(\partial_2\hat{\mathcal{H}}_2 + \partial_3\hat{\mathcal{H}}_3 ) (\hat{v}_2\mathfrak{H}_{2}^{\varepsilon} + \hat{v}_3\mathfrak{H}_{3}^{\varepsilon}   )\big\}
\quad \mbox{on  }{\omega}
\, .
\end{multline}
From \eqref{19a0}, \eqref{AAA} we obtain (we restore the index $\gamma$)
\begin{multline}
\gamma\left( \|\mathcal{U}_\gamma^{\varepsilon}\|^2_{L^2(Q^+)}+\|W_\gamma^{\varepsilon}\|^2_{L^2(Q^-)}\right) 
 \leq \frac{C}{\gamma}\left\{\|\mathcal{F}_\gamma\|^2_{L^2(Q^+)}
 + \|\mathcal{U}_{n\gamma}^{\varepsilon}|_\omega\|^2_{L^2(\omega)}+ \|W_{\gamma}^{\varepsilon}|_\omega\|^2_{L^2(\omega)}
 \right\}
 \\
 +C\left( \|\mathcal{U}_\gamma^{\varepsilon}\|^2_{L^2(Q^+)}+\|W_\gamma^{\varepsilon}\|^2_{L^2(Q^-)}\right)
 +  \gamma \|\varphi_{\gamma}^{\varepsilon}\|^2_{L^2(\omega)},\label{L2a}
\end{multline}
where $C$ is independent of ${\varepsilon},\gamma$. Thus if $\gamma_0$ is  large enough we obtain from \eqref{phi}, \eqref{L2a} the inequality
\begin{multline}
\gamma\left( \|\mathcal{U}_\gamma^{\varepsilon}\|^2_{L^2(Q^+)}+\|W_\gamma^{\varepsilon}\|^2_{L^2(Q^-)}\right) \\
 \leq \frac{C}{\gamma}\left\{\|\mathcal{F}_\gamma\|^2_{L^2(Q^+)}
 + \|\mathcal{U}_{n\gamma}^{\varepsilon}|_\omega\|^2_{L^2(\omega)}+\|W_{\gamma}^{\varepsilon}|_\omega\|^2_{L^2(\omega)}
 \right\}
,
 \quad  0<\varepsilon<\varepsilon_0,\; \gamma\ge \gamma_0,
\label{L2b}
\end{multline}
where $C$ is independent of ${\varepsilon},\gamma$.

%%%%%%%%%%%%%%%%%%%%%%%%%%%%%%%%%%%%%%%%%%%%%%%%%
%%%%%%%%%%%%%%%%%%%%%%%%%%%%%%%%%%%%%%%%%%%%%%%%%
%%%%%%%%%%%%%%%%%%%%%%%%%%%%%%%%%%%%%%%%%%%%%%%%%
%%%%%%%%%%%%%%%%%%%%%%%%%%%%%%%%%%%%%%%%%%%%%%%
%%%%%%%%%%%%%%%%%%%%%%%%%%%%%%%%%%%%%%%%%%%%%%%
%%%%%%%%%%%%%%%%%%%%%%%%%%%%%%%%%%%%%%%%%%%%%%%
%%%%%%%%%%%%%%%%%%%%%%%%%%%%%%%%%%%%%%%%%
%%%%%%%%%%%%%%%%%%%%%%%%%%%%%%%%%%%%%%%%%
Now we derive the a priori estimate of tangential derivatives. Differentiating systems \eqref{77'a} and \eqref{50"} with respect to $x_0=t$, $x_2$ or $x_3$, using standard arguments of the energy method, and applying \eqref{130}, \eqref{127}, gives the energy inequality
\begin{multline}
\gamma\int_{Q^+}(\widehat{A}_0Z_{\ell}U^{\varepsilon}_\gamma,Z_{\ell}U^{\varepsilon}_\gamma)\,dxdt+
\gamma\int_{Q^-}(M_0^{\varepsilon}Z_{\ell}W^{\varepsilon}_\gamma,Z_{\ell}W_\gamma^{\varepsilon})dxdt
+ \int_{\omega}\mathcal{A}^{\varepsilon}_{\ell}\,dx'dt  
\\
\leq \frac{C}{\gamma}
\left\{ \|\mathcal{F}_\gamma\|^2_{H^1_{tan,\gamma}(Q^+)}+\|\mathcal{U}_\gamma^{\varepsilon}\|^2_{H^1_{tan,\gamma}(Q^+)} +\|W_\gamma^{\varepsilon}\|^2_{H^1_{\gamma}(Q^-)} \right\},
\label{19a}
\end{multline}
where $\ell =0,2,3$, and where we have denoted
\begin{equation*}
\begin{array}{ll}\label{19b}
\ds \mathcal{A}^{\varepsilon}_{\ell}=-\frac{1}{2}(\mathcal{E}_{12}Z_{\ell}U^{\varepsilon}_{\gamma},Z_{\ell}
U^{\varepsilon}_{\gamma})|_{\omega}+\frac{1}{2}(M_1^{\varepsilon}Z_{\ell}W^{\varepsilon}_{\gamma},Z_{\ell}
W^{\varepsilon}_{\gamma})|_{\omega}\,.

\end{array}
\end{equation*}
Thanks to the properties of the matrices $M_{\alpha}^{\varepsilon}$ ($\alpha =\overline{0,4}$) described in \eqref{ordine}, the constant $C$ in \eqref{19a} is uniformly bounded in ${\varepsilon}$ and $\gamma$.
%
%%%%%%%%%%%%%%%%%%%%%%%%%%%%%%%%%%%%%%%%%
%%%%%%%%%%%%%%%%%%%%%%%%%%%%%%%%%%%%%%%%%
We repeat for $\ds \mathcal{A}^{\varepsilon}_{\ell}$ the calculations leading to \eqref{AAA} for $\ds \mathcal{A}^{\varepsilon}$. Clearly, for the same choices as in \eqref{nu} and \eqref{E} we obtain (for simplicity we drop again the index $\gamma$)
\begin{multline}
\mathcal{A}_{\ell}^{\varepsilon}=
Z_{\ell}\varphi^{\varepsilon} \big\{ [ \partial_1\hat{q}]\,(Z_{\ell}u_1^{\varepsilon}+Z_{\ell}\varphi^{\varepsilon}\partial_1\hat{v}_N) + \partial_1\hat{v}_NZ_{\ell}q^{\varepsilon}
\\
+ (  \partial_t\hat{\mathcal{H}}_3 - \partial_2\hat{E}_1)(Z_{\ell}\mathfrak{H}^{\varepsilon}_{3} + \varepsilon\hat{v}_2 Z_{\ell}E_N^{\varepsilon})
+ ( \partial_t\hat{\mathcal{H}}_2 + \partial_3\hat{E}_1)(Z_{\ell}\mathfrak{H}^{\varepsilon}_{2} - \varepsilon\hat{v}_3 Z_{\ell}E_N^{\varepsilon})
\\+(\partial_2\hat{\mathcal{H}}_2 + \partial_3\hat{\mathcal{H}}_3 ) (\hat{v}_2Z_{\ell}\mathfrak{H}_{2}^{\varepsilon} + \hat{v}_3Z_{\ell}\mathfrak{H}_{3}^{\varepsilon}   )\big\}
+{\rm l.o.t.}\,\! ,
\quad \mbox{on  }{\omega},
\label{Al}
\end{multline}
where l.o.t. is the sum of lower-order terms.
Using \eqref{12a} we reduce the above terms to those like
\[
\hat{c}\,h_1^{\varepsilon}Z_{\ell}u_1^{\varepsilon},\quad \hat{c}\, {h}_1^{\varepsilon}Z_{\ell}\varphi^{\varepsilon},
\quad
\hat{c}\, {h}_1^{\varepsilon}Z_{\ell}\mathfrak{H}_j^{\varepsilon},\quad \hat{c}\, {h}_1^{\varepsilon}Z_{\ell}\mathfrak{E}_j^{\varepsilon}
,\quad\dots
\quad \mbox{on  }{\omega}
,\]
terms as above with $\mathfrak{h}_1^{\varepsilon},u_1^{\varepsilon}$ instead of $h_1^{\varepsilon}$, or even ``better'' terms like
$$
\gamma\hat{c}\varphi^{\varepsilon} Z_{\ell}{u}_1^{\varepsilon},\quad
\gamma\hat{c}\varphi^{\varepsilon} Z_{\ell}\varphi^{\varepsilon}.$$ Here and below $\hat{c}$ is the common notation for a generic coefficient depending on the basic state \eqref{21}.
By integration by parts such ``better'' terms can be reduced to the above ones and terms of lower order.
%%%%%%%%

The terms like $\hat{c}\,h_1^{\varepsilon}Z_{\ell} u^{\varepsilon}_{1|x_1=0}$ are estimated by passing to the volume integral and integrating by parts:
\begin{multline*}
\int_{\omega}\hat{c}\,h_1^{\varepsilon}Z_{\ell}u^{\varepsilon}_{1|x_1=0}\,{\rm d}x'\,{\rm d}t
=-\int_{Q^+}\partial_1\bigl(\tilde{c}h_1^{\varepsilon}Z_{\ell}u_1^{\varepsilon}\bigr){\rm d}x\,{\rm d}t \\
=\int_{Q^+}\Bigl\{(Z_{\ell}\tilde{c})h_1^{\varepsilon}(\partial_1u_1^{\varepsilon})+\tilde{c}(Z_{\ell}h_1^{\varepsilon})\partial_1u_1^{\varepsilon}  -(\partial_1\tilde{c})h_1^{\varepsilon}
Z_\ell u_1^{\varepsilon}
-\tilde{c}(\partial_1h_1^{\varepsilon})Z_{\ell}u_1^{\varepsilon}
\Bigr\}{\rm d}x\,{\rm d}t
,
\end{multline*}
where $\tilde{c}|_{x_1=0}=\hat{c}$. Estimating the right-hand side by the H\"older's inequality and \eqref{130} gives
\begin{equation}
\begin{array}{ll}\label{unacaso}
\ds\left|
\int_{\omega}\hat{c}\,h_1^{\varepsilon}Z_{\ell}u^{\varepsilon}_{1|x_1=0}\,{\rm d}x'\,{\rm d}t
\right| \leq
  C\left\{ \|\mathcal{F}_\gamma\|^2_{L^2(Q^+)} +\|\mathcal{U}_\gamma^{\varepsilon}\|^2_{H^1_{tan,\gamma}(Q^+)} \right\}.
\end{array}
\end{equation}
%%%%%%%%%%%%%%%%
In the same way we estimate the other similar terms $  \hat{c}\, {h}_1^{\varepsilon}Z_{\ell}\mathfrak{H}_j^{\varepsilon}, \hat{c}\, {h}_1^{\varepsilon}Z_{\ell}\mathfrak{E}_j^{\varepsilon}
,$ etc. Notice that we only need to estimate normal derivatives either of components of $\mathcal{U}_{n\gamma}^{\varepsilon}$ or $W_{\gamma}^{\varepsilon}$. For terms like $\hat{c}\,{\mathfrak{h}}^{\varepsilon}_{1}Z_{\ell}u_1^{\varepsilon}, \hat{c}\,{\mathfrak{h}}^{\varepsilon}_{1}Z_{\ell}\mathfrak{E}_j^{\varepsilon}$, etc. we use \eqref{127} instead of \eqref{130}.

%%%%%%%%%%%%%
We treat the terms like $\hat{c}\,h_{1|x_1=0}^{\varepsilon}Z_{\ell} \varphi^{\varepsilon}$  by substituting \eqref{12a} again:
\begin{multline}\label{unacaso2}
\left|\int_{\omega}\hat{c}\,h_1^{\varepsilon}Z_{\ell}\varphi^{\varepsilon}\,{\rm d}x'\,{\rm d}t\right|
=\left|\int_{\omega}\hat{c}\,h_1^{\varepsilon}\Bigl(\hat{a}_1{h}^{\varepsilon}_{1}+\hat{a}_2{\mathfrak{h}}^{\varepsilon}_{1} +\hat{a}_3{u}^{\varepsilon}_{1}+\hat{a}_4\varphi^{\varepsilon}+\gamma\hat{a}_5\varphi^{\varepsilon}
\Bigr){\rm d}x\,{\rm d}t \right|
\\
\le
C\left( \|\mathcal{U}_{n}^{\varepsilon}|_\omega\|^2_{L^2(\omega)}+\|W^{\varepsilon}|_\omega\|^2_{L^2(\omega)} +\gamma^2\|\varphi^{\varepsilon}\|^2_{L^2(\omega)}\right).
\end{multline}
%%%%%%%%%%%%%

Combining \eqref{19a}, \eqref{unacaso}, \eqref{unacaso2} and similar inequalities for the other terms of \eqref{Al} yields (we restore the index $\gamma$)
\begin{multline}
\gamma\left( \|Z_{\ell}\mathcal{U}_\gamma^{\varepsilon}\|^2_{L^2(Q^+)}+\|Z_{\ell}W_\gamma^{\varepsilon}\|^2_{L^2(Q^-)}\right) 
 \leq C\Big\{\frac{1}{\gamma}\|\mathcal{F}_\gamma\|^2_{H^1_{tan,\gamma}(Q^+)}
 +\|\mathcal{U}_\gamma^{\varepsilon}\|^2_{H^1_{tan,\gamma}(Q^+)}
 +\|W_\gamma^{\varepsilon}\|^2_{H^1_{\gamma}(Q^-)}
\\
 +\gamma \left( \|\mathcal{U}_{n\gamma}^{\varepsilon}|_\omega\|^2_{L^2(\omega)} +  \|W_{\gamma}^{\varepsilon}|_\omega\|^2_{L^2(\omega)} \right)
\Big\},
 \quad  0<\varepsilon<\varepsilon_0,\; \gamma\ge \gamma_0,
\label{Zell}
\end{multline}
where $C$ is independent of ${\varepsilon},\gamma$. Then from \eqref{126''}, \eqref{127}, \eqref{L2b}, \eqref{Zell} we obtain
%%%%
\begin{multline}
\gamma\left( \|\mathcal{U}_\gamma^{\varepsilon}\|^2_{H^1_{tan,\gamma}(Q^+)}
+\|W_\gamma^{\varepsilon}\|^2_{H^1_{\gamma}(Q^-)}\right) 
 \leq C\Big\{\frac{1}{\gamma}\|\mathcal{F}_\gamma\|^2_{H^1_{tan,\gamma}(Q^+)}
 +\|\mathcal{U}_\gamma^{\varepsilon}\|^2_{H^1_{tan,\gamma}(Q^+)}
 +\|W_\gamma^{\varepsilon}\|^2_{H^1_{\gamma}(Q^-)}
\\ +\gamma \left( \|\mathcal{U}_{n\gamma}^{\varepsilon}|_\omega\|^2_{L^2(\omega)} +  \|W_{\gamma}^{\varepsilon}|_\omega\|^2_{L^2(\omega)} \right)
\Big\},
 \quad  0<\varepsilon<\varepsilon_0,\; \gamma\ge \gamma_0,
\label{H1}
\end{multline}
where $C$ is independent of ${\varepsilon},\gamma$. 
%
%%%%%%%%%%%%%%%%%%%%%%%%%%%%%%%%%%%%%%%%%%%%
We need the following estimates for the trace of $\mathcal{U}^{\varepsilon}_n,W^{\varepsilon}$.
\begin{lemma}\label{interpol}
The functions $\mathcal{U}^{\varepsilon}_n,W^{\varepsilon}$ satisfy
\begin{equation}
\begin{array}{ll}\label{interpol1}
\gamma\|\mathcal{U}_{n\gamma}^{\varepsilon}|_\omega\|^2_{L^2(\omega)}
+
\|\mathcal{U}_{n\gamma}^{\varepsilon}|_\omega\|^2_{H^{1/2}_{\gamma}(\omega)}
\le C\left(\|\mathcal{F}_\gamma\|^2_{L^2(Q^+)}
 +\|\mathcal{U}_\gamma^{\varepsilon}\|^2_{H^1_{tan,\gamma}(Q^+)}
\right),

\end{array}
\end{equation}
\begin{equation}
\begin{array}{ll}\label{interpol2}

\gamma  \|W_{\gamma}^{\varepsilon}|_\omega\|^2_{L^2(\omega)}
+
\|W_{\gamma}^{\varepsilon}|_\omega\|^2_{H^{1/2}_{\gamma}(\omega)}
\le C  \|W_\gamma^{\varepsilon}\|^2_{H^1_{\gamma}(Q^-)}.

\end{array}
\end{equation}
\end{lemma}
The proof of Lemma \ref{interpol} is given in Section \ref{proofinterpol}
at the end of this article. Substituting \eqref{interpol1}, \eqref{interpol2} in \eqref{H1} and taking $\gamma_0$ large enough yields
%%%%
\begin{equation}
\gamma\left( \|\mathcal{U}_\gamma^{\varepsilon}\|^2_{H^1_{tan,\gamma}(Q^+)}
+\|W_\gamma^{\varepsilon}\|^2_{H^1_{\gamma}(Q^-)}\right) 
 \leq \frac{C}{\gamma}\|\mathcal{F}_\gamma\|^2_{H^1_{tan,\gamma}(Q^+)}
, \quad  0<\varepsilon<\varepsilon_0,\; \gamma\ge \gamma_0,
\label{H11}
\end{equation}
where $C$ is independent of ${\varepsilon},\gamma$. Finally, from
\eqref{gradphi}, \eqref{interpol1} and \eqref{H11} we get
\begin{equation}
\begin{array}{ll}\label{gradphi2}
\ds \gamma \left( \|\mathcal{U}_{n\gamma}^{\varepsilon}|_\omega\|^2_{H^{1/2}_{\gamma}(\omega)}+ \|W_{\gamma}^{\varepsilon}|_\omega\|^2_{H^{1/2}_{\gamma}(\omega)} \right)
+ \gamma^2 \|\varphi^{\varepsilon}\|^2_{H^{1}_{\gamma}(\omega)}
 \leq \frac{C}{\gamma}\|\mathcal{F}_\gamma\|^2_{H^1_{tan,\gamma}(Q^+)}
.
\end{array}
\end{equation}
Adding \eqref{H11}, \eqref{gradphi2} gives \eqref{54'}, The proof of Theorem \ref{lem1} is complete.

%%%%%%%%%%%%%%%%%%%%%%%%%%%%%%%%%%%%%%%%%%%%
%%%%%%%%%%%%%%%%%%%%%%%%%%%%%%%%%%%%%%%%%%%%
%%%%%%%%%%%%%%%%%%%%%%%%%%%%%%%%%%%%%%%%%%%%
%%%%%%%%%%%%%%%%%%%%%%%%%%%%%%%%%%%%%%%%%%%%
%%%%%%%%%%%%%%%%%%%%%%%%%%%%%%%%%%%%%%%%%%%%
%%%%%%%%%%%%%%%%%%%%%%%%%%%%%%%%%%%%%%%%%%%%
%%%%%%%%%%%%%%%%%%%%%%%%%%%%%%%%%%%%%%%%%%%%
\section{Well-posedness of the hyperbolic regularized problem}\label{well}

In this section we prove the existence of the solution of \eqref{77'}. Its restriction to the time interval $(-\infty,T]$ will provide the solution of problem \eqref{77}. From now on, in the proof of the existence of the solution, $\varepsilon$ is fixed and so we omit it and we simply write $\mathcal{U}$ instead of $\mathcal{U}^{\varepsilon}$, $W$ instead of $W^{\varepsilon}$, $\varphi$ instead of $\varphi^{\varepsilon}$.

In view of the result of Lemma \ref{equivalence} (see Section \ref{equival}) we can consider system \eqref{50"} instead of \eqref{78'}.
First of all, we write the boundary conditions in different form, by eliminating the derivatives of $\varphi$.
We substitute \eqref{79'} in the boundary conditions for $\mathfrak{E}_2,\mathfrak{E}_3$ and take account of the constraint \eqref{87} and the choices \eqref{nu}, \eqref{E}. We get
%%%%%%%%%%%%%%%%%%%%%%%%%%%%%%%%%%%%%%%%%%%%%%%%
\begin{equation}
\begin{array}{ll}\label{108}
{q} -\hat{\mathfrak{h}}_2\mathfrak{H}_2-\hat{\mathfrak{h}}_3\mathfrak{H}_3 +\varepsilon\,
\hat{E}_1 {E}_N+  [ \partial_1\hat{q}]\varphi=0,  \qquad&\\
\mathfrak{E}_2-\varepsilon\,\widehat{\mathcal{H}}_3u_1 +\varepsilon\,\hat{v}_3\mathcal{H}_N +\varepsilon a_1\varphi=0, \qquad& \\
\mathfrak{E}_3+\varepsilon\,\widehat{\mathcal{H}}_2u_1 -\varepsilon\,\hat{v}_2\mathcal{H}_N +\varepsilon a_2\varphi=0, \qquad& 
\mbox{on}\ \omega,
\end{array}
\end{equation}
%%%%%%%%%%%%%%%%%%%%%%%%%%%%%%%%%%%%%%%%%%%%%%%%
where the precise form of the coefficients $a_1,a_2$ is not important. For later use we observe that \eqref{87},  \eqref{79'}-\eqref{82'} is equivalent to \eqref{87}, \eqref{79'}, \eqref{108}. Notice that the last two equations in \eqref{108}
yield
\begin{equation}
\begin{array}{ll}\label{109}
\varepsilon\, \hat{E}_1 u_1 +\hat{v}_2 \mathfrak{E}_2+\hat{v}_3 \mathfrak{E}_3 +\varepsilon a_3\varphi=0,
\end{array}
\end{equation}
where $a_3=a_1\hat{v}_2 +a_2\hat{v}_3$.

%%%%%%%%%%%%%%%%%%%%%%%%%%%%%%%%%%%%%%%%%%%%%%%%%%%%%
%%%%%%%%%%%%%%%%%%%%%%%%%%%%%%%%%%%%%%%%%%%%%%%%%%%%%
%%%%%%%%%%%%%%%%%%%%%%%%%%%%%%%%%%%%%%%%%%%%%%%%%%%%%
%%%%%%%%%%%%%%%%%%%%%%%%%%%%%%%%%%%%%%%%%%%%%%%%%
%%%%%%%%%%%%%%%%%%%%%%%%%%%%%%%%%%%%%%%%%%%%%%%%%
%%%%%%%%%%%%%%%%%%%%%%%%%%%%%%%%%%%%%%%%%%%%%%%%%
%%%%%%%%%%%%%%%%%%%%%%%%%%%%%%%%%%%%%%%%%%%%%%%%%%%%%
%\subsection{The variational formulation}

Let us write the system \eqref{77'a}, \eqref{50"},  \eqref{108} in compact form as
\begin{equation}
\begin{array}{ll}\label{compactform2}
\begin{cases}
\mathcal{L}\left(\begin{array}{c}
\mathcal{U} \\
W
 \end{array} \right)=\left(\begin{array}{c}
\mathcal{F} \\
0
 \end{array} \right) &\mbox{on}\; {Q^+}\times {Q^-},
\\
M
\left(\begin{array}{c}
\mathcal{U} \\
W
 \end{array} \right)+b\,\varphi =0,
 &\mbox{in}\;\omega,
\\
(\mathcal{U},W,\varphi )=0\qquad &\mbox{for}\ t<0,
\end{cases}\end{array}
\end{equation}
where the matrix $M$ and the vector $b$ are implicitly defined by \eqref{108}.

%%%%%%%%%%%%%%%%%%%%%%%%%%%%%%%%%%%%%%%%%%%%%%%%%%%%%

Let us multiply \eqref{compactform2} by $e^{-\gamma t}$ with $\gamma\ge1$; according to the rule $e^{-\gamma t}\dt u=(\gamma+\dt)e^{-\gamma t}u$, \eqref{compactform2} becomes equivalent to
\begin{equation}
\begin{array}{ll}\label{compactform3}
\begin{cases}
\mathcal{L}_{\gamma}\left(\begin{array}{c}
\mathcal{U}_{\gamma} \\
W_{\gamma}
 \end{array} \right)=\left(\begin{array}{c}
\mathcal{F}_{\gamma} \\
0
 \end{array} \right) &\mbox{on}\; {Q^+}\times {Q^-},
 \\
M
\left(\begin{array}{c}
\mathcal{U}_{\gamma} \\
W_{\gamma}
 \end{array} \right)+b\,\varphi_{\gamma} =0
 &\mbox{in}\;\omega,
\\
(\mathcal{U}_{\gamma},W_{\gamma},\varphi_{\gamma} )=0\qquad &\mbox{for}\ t<0.
\end{cases}\end{array}
\end{equation}
where
\[
\mathcal{L}_{\gamma}:=\gamma 
\left(\begin{array}{cc}
\hat{\mathcal{A}}_0&0 \\
0&M_0^\varepsilon
 \end{array} \right)+\mathcal{L},
\]
$\mathcal{U}_{\gamma}=e^{-\gamma t}\,\mathcal{U}, W_{\gamma}=e^{-\gamma t}\, W, \varphi_{\gamma}=e^{-\gamma t}\varphi$, etc..
%%%%%%%%%%%%%%%%%%%%%%%%%%%%%%%%%%%%%%%%%%%%%%%%%%%%%
%%%%%%%%%%%%%%%%%%%%%%%%%%%%%%%%%%%%%%%%%%%%%%%%%%%%%
%%%%%%%%%%%%%%%%%%%%%%%%%%%%%%%%%%%%%%%%%%%%%%%%%%%%%
%%%%%%%%%%%%%%%%%%%%%%%%%%%%%%%%%%%%%%%%%%%%%%%%%%%%%
%%%%%%%%%%%%%%%%%%%%%%%%%%%%%%%%%%%%%%%%%%%%%%%%%%%
%\begin{remark}\label{regbdry}
%If $W_{\gamma}$ is a solution in the above sense then it solves \eqref{78'} in the sense of distributions, and therefore the two constraints \eqref{85} as well.
%From $(\mathcal{U}_{\gamma},W_{\gamma})\in L^2(Q^+)\times L^2(Q^-)$ with $\mathcal{L}_{\gamma}(\mathcal{U}_{\gamma},
%W_{\gamma})\in L^2(Q^+)\times L^2(Q^-)$ one gets a control of the noncharacteristic part at the boundary $(q_\gamma, u_{1\gamma})_{|\omega} \in  H^{-1/2}_\gamma(\omega)$, $ (\mathfrak{H}_{2\gamma},  \mathfrak{H}_{3\gamma},\mathfrak{E}_{2\gamma}, \mathfrak{E}_{3\gamma})_{|\omega} \in  H^{-1/2}_\gamma(\omega)$. Then, from \eqref{85}, one obtains $(\mathfrak{h}_{1\gamma}, \mathfrak{e}_{1\gamma})_{|\omega} \in  H^{-1/2}_\gamma(\omega)$, which all together implies  $(\mathfrak{H}_{\gamma},\mathfrak{E}_{\gamma})_{|\omega} \in  H^{-1/2}_\gamma(\omega)$. 
%It follows that weak solutions of \eqref{compactform3} satisfy the boundary conditions  \eqref{108} in the sense of $H^{-1/2}_\gamma(\omega)$.
%\end{remark}
%%%%%%%%%%%%%%%%%%%%%%%%%%%%%%%%%%%%%%%%%%%%%%%%%%%
%%%%%%%%%%%%%%%%%%%%%%%%%%%%%%%%%%%%%%%%%%%%%%%
%%%%%%%%%%%%%%%%%%%%%%%%%%%%%%%%%%%%%%%%%%%%%%%

%%%%%%%%%%%%%%%%%%%%%%%%%%%%%%%%%%%%%%%%%%%%%%%
%%%%%%%%%%%%%%%%%%%%%%%%%%%%%%%%%%%%%%%%%%%%%%%
%%%%%%%%%%%%%%%%%%%%%%%%%%%%%%%%%%%%%%%%%%%%%%%
First we solve \eqref{compactform3} under the assumption that $\varphi_{\gamma}$ is given.
\begin{lemma}\label{lemmaH1}
There exists $\gamma_0>0$ such that for all $\gamma\ge\gamma_0$ and for all given $\mathcal{F} \in e^{\gamma t} H^1_{tan,\gamma}(Q^+)$ and
$\varphi\in e^{\gamma t}H^{3/2}_\gamma(\omega)$ vanishing in the past,  the problem 
\eqref{compactform3} has a unique solution $(\mathcal{U},W)\in e^{\gamma t} H^1_{tan,\gamma}(Q^+)\times e^{\gamma t} H^1_\gamma(Q^-)$ with $(q,u_{ 1}h_1,W)_{|\omega}\in e^{\gamma t}{H^{1/2}_\gamma(\omega)}$, such that
\begin{multline}
\label{estimsol2}
\ds \|e^{-\gamma t}\,(\mathcal{U},W)\|_{H^1_{tan,\gamma}(Q^+) \times H^1_\gamma(Q^-)}
+\|e^{-\gamma t}(q,u_{1},h_{1},W_{})_{|\omega}\|_{H^{1/2}_\gamma(\omega)}
\\
\ds \leq 
\frac{C}{\gamma}  \left( \|e^{-\gamma t}\,\mathcal{F} \|_{H^1_{tan,\gamma}(Q^+) }+\|e^{-\gamma t}\,\varphi\|_{H^{3/2}_\gamma(\omega)} \right)
.
\end{multline}

\end{lemma}
\begin{proof}
We insert the new boundary conditions \eqref{108}, \eqref{109} in 
the quadratic form $\mathcal{A}^{\varepsilon}$ (see \eqref{A}) and we get
\begin{equation}
\label{110}
\ds \mathcal{A}^{\varepsilon}:=-\frac{1}{2}(\hat{\mathcal{A}}_1+\mathcal{E}_{12})\,\mathcal{U}\cdot\mathcal{U} + \frac{1}{2}{M}_1^{\varepsilon}W\cdot W
=\left( [\partial\hat{q}]u_1 + a_2 \mathfrak{H}_2-a_1 \mathfrak{H}_3 - {\varepsilon}a_3 E_N \right) \varphi
\quad\mbox{on }\omega
.
\end{equation}
If we consider the boundary conditions \eqref{108}, \eqref{109} in homogeneous form, namely if we set $\varphi=0$, then from \eqref{110}
$$\mathcal{A}^{\varepsilon}=0 \qquad\mbox{on }\omega
.$$
We deduce that the boundary conditions \eqref{108} are nonnegative for ${\mathcal{L}}_\gamma$. 
As the number of boundary conditions in \eqref{108} is in agreement with the number of incoming characteristics for the operator ${\mathcal{L}}_\gamma$ (see Proposition \ref{nrbc}) we infer that the boundary conditions \eqref{108} are maximally nonnegative (but not strictly dissipative). 
Then we reduce the problem to one with homogeneous boundary conditions by subtracting from $(\mathcal{U}_{\gamma},W_{\gamma})$ a function $(\mathcal{U}'_{\gamma},W'_{\gamma})\in H^2_\gamma(Q^+) \times H^2_\gamma(Q^-)$ such that 
\begin{equation*}
\begin{array}{ll}\label{newbc2}
M\left(\begin{array}{c}
\mathcal{U}' \\
W'
 \end{array} \right)+b\,\varphi =0 \qquad\mbox{on }\omega
.
\end{array}
\end{equation*}
Finally, as the boundary is characteristic of constant multiplicity \cite{rauch85}, we may apply the result of \cite{secchi95,secchi96} and we get the solution with the prescribed regularity.
\end{proof}

%%%%%%%%%%%%%%%%%%%%%%%%%%%%%%%%%%%%%%%%%%%%%%%
%%%%%%%%%%%%%%%%%%%%%%%%%%%%%%%%%%%%%%%%%%%%%%%
%%%%%%%%%%%%%%%%%%%%%%%%%%%%%%%%%%%%%%%%%%%%%%%
%%%%%%%%%%%%%%%%%%%%%%%%%%%%%%%%%%%%%%%%%%%%%%%
The well-posedness of \eqref{77'} in $H^1_{tan}\times H^1$ is given by the following theorem.
\begin{theorem}\label{existenceH1}
There exists $\gamma_0>0$ such that for all $\gamma\ge\gamma_0$ and $\mathcal{F} \in e^{\gamma t} H^1_{tan,\gamma}(Q^+)$ vanishing in the past, the problem 
\eqref{77'} has a unique solution $(\mathcal{U},W)\in e^{\gamma t} H^1_{tan,\gamma}(Q^+)\times e^{\gamma t} H^1_{\gamma}(Q^-)$ with $(q,u_{ 1}h_1,W)_{|\omega}\in e^{\gamma t}{H^{1/2}_\gamma(\omega)}$,
$\varphi\in e^{\gamma t}H^{3/2}_\gamma(\omega)$.

\end{theorem}
\begin{proof}
We prove the existence of the solution to \eqref{77'} by a fixed point argument. Let $\overline\varphi\in e^{\gamma t} H^{3/2}_{\gamma}(\omega_T)$ vanishing in the past. By Lemma \ref{lemmaH1}, for $\gamma$ sufficiently large there exists a unique solution  $(\mathcal{U},W)\in e^{\gamma t} H^1_{tan,\gamma}(Q^+)\times e^{\gamma t} H^1_\gamma(Q^-)$, with  $(q,u_1,h_1,W)_{|\omega} \in  e^{\gamma t} H^{1/2}_\gamma(\omega)$
of
\begin{equation}
\begin{array}{ll}\label{equa1}
\begin{cases}
\mathcal{L}_{\gamma}\left(\begin{array}{c}
\mathcal{U}_\gamma \\
W_\gamma
 \end{array} \right)=\left(\begin{array}{c}
\mathcal{F}_{\gamma} \\
0
 \end{array} \right) &\mbox{on}\; {Q^+}\times {Q^-},
 \\
M
\left(\begin{array}{c}
\mathcal{U}_\gamma \\
W_\gamma
 \end{array} \right)= -b\,\overline\varphi_\gamma
&\mbox{on}\;\omega,
\\
(\mathcal{U}_\gamma,W_\gamma)=0\qquad &\mbox{for}\ t<0,
\end{cases}\end{array}
\end{equation}
enjoying the a priori estimate
\eqref{estimsol2} with $\overline\varphi$ instead of $\varphi$.
%%%%%%%%%%%%%%%%%%%%%%%%%%%%%%%%%%%%%%%%%%%%%%%
%%%%%%%%%%%%%%%%%%%%%%%%%%%%%%%%%%%%%%%%%%%%%%%
Now consider the equation
\begin{equation}
\label{equa2}
{\gamma}\varphi_\gamma+\partial_t\varphi_\gamma+\hat{v}_2\partial_2\varphi_\gamma+\hat{v}_3\partial_3\varphi_\gamma-
\varphi_\gamma\partial_1\hat{v}_{N}=u_{1\gamma}
 , \qquad \mbox{on}\;\omega,
\end{equation}
where $u_{1\gamma} \in{H^{1/2}_{\gamma}(\omega)}$ is the trace of the component of  $\mathcal{U}_\gamma$ given in the previous step, vanishing for $t<0$. 
For $\gamma$ sufficiently large  there exists a unique solution $\varphi_{\gamma}\in H^{1/2}_{\gamma}(\omega)$, vanishing in the past, such that
\begin{equation}\label{129}
\ds\|\varphi_\gamma\|_{H^{1/2}_{\gamma}(\omega)}
\le\frac{C}{\gamma} \|u_{1\gamma} \|_{H^{1/2}_{\gamma}(\omega)}.
\end{equation}
%%%%%%%%%%%%%%%%%%%%%%%%%%%%%%%%%%%%%%%%%%%%%%%
%%%%%%%%%%%%%%%%%%%%%%%%%%%%%%%%%%%%%%%%%%%%%%%
From the plasma equation in \eqref{equa1} and from \eqref{equa2} we deduce the boundary constraint
\begin{equation}
\begin{array}{ll}\label{equa3}

{h}_{1\gamma}=\widehat{H}_2\partial_2\varphi_{\gamma} +\widehat{H}_3\partial_3\varphi_{\gamma} -
\varphi_{\gamma}\partial_1\widehat{H}_{N}\qquad \mbox{on}\ \omega.
\end{array}
\end{equation}
Since in the right-hand side of \eqref{equa1} we have $\overline\varphi$ instead of $\varphi$ we are not able to deduce the similar boundary constraint for the vacuum magnetic field. Instead, we obtain
\begin{equation}
\begin{array}{ll}\label{equa4}
{\mathfrak{h}}_{1\gamma} -\partial_2\bigl(\widehat{\mathcal{H}}_2\varphi_{\gamma} \bigr) -\partial_3\bigl(\widehat{\mathcal{H}}_3\varphi_{\gamma} \bigr)=G_{\gamma}
\qquad \mbox{on}\ \omega,

\end{array}
\end{equation}
where $G_{\gamma}$ solves 
\begin{equation}
\begin{array}{ll}\label{equaG}
\ds \frac{\tilde{d}}{dt}G_{\gamma}+a_2\partial_2(\varphi_{\gamma} -\overline\varphi_{\gamma})-a_1\partial_3(\varphi_{\gamma} -\overline\varphi_{\gamma}) + (\partial_2a_2-\partial_3a_1)(\varphi_{\gamma} -\overline\varphi_{\gamma})=0
\qquad \mbox{on}\ \omega,
\end{array}
\end{equation}
for $\tilde{d}/dt=\gamma+\partial_t+\partial_2(\hat{v}_2\cdot)+\partial_3(\hat{v}_3\cdot)$ and where the coefficients $a_1,a_2$ are the same of \eqref{108}. \eqref{equaG} is derived from the first equation of the vacuum part of \eqref{equa1}, \eqref{equa2} and the boundary conditions for $\mathfrak{E}_2,\mathfrak{E}_3$ in \eqref{equa1}.

%%%%%%%%%%%%%%%%%%%%%%%%%%%%%%%%%%%%%%%%%%%%%%%
%%%%%%%%%%%%%%%%%%%%%%%%%%%%%%%%%%%%%%%%%%%%%%%
Let us consider the linear system for $\nabla_{t,x'}\varphi_\gamma$ provided by equations \eqref{equa2}, \eqref{equa3} and \eqref{equa4}. By the stability condition \eqref{41} we can express $\nabla_{t,x'}\varphi_\gamma$ through $(h_{1\gamma}, \mathfrak{h}_{1\gamma}, u_{1\gamma})_{|\omega}, \varphi_\gamma, G_\gamma$, that is
\begin{equation}
\begin{array}{ll}\label{equaphi}
\nabla_{t,x'}\varphi_\gamma=a_1'h_{1\gamma}+a_2'\mathfrak{h}_{1\gamma}+a_3'u_{1\gamma}+a_4'\varphi_\gamma+a_5'G_\gamma,

\end{array}
\end{equation}
where the precise form of the coefficients $a_i'$ has no interest. Then, substituting into \eqref{equaG} yields
\begin{equation}
\begin{array}{ll}\label{equaG2}
\ds \frac{\tilde{d}}{dt}G_{\gamma}+b_0G_{\gamma}=b_1h_{1\gamma}+b_2\mathfrak{h}_{1\gamma}+b_3\varphi_\gamma+
a_2\partial_2\overline\varphi_{\gamma}-a_1\partial_3\overline\varphi_{\gamma} + (\partial_2a_2-\partial_3a_1)\overline\varphi_{\gamma}
\qquad \mbox{on}\ \omega,
\end{array}
\end{equation}
with suitable coefficients $b_i$.

%%%%%%%%%%%%%%%%%%%%%%%%%%%%%%%%%%%%%%%%%%%%%%%
%%%%%%%%%%%%%%%%%%%%%%%%%%%%%%%%%%%%%%%%%%%%%%%
From \eqref{equaG2}, for $\gamma$ sufficiently large, we get the estimate
\begin{multline}\label{equaG3}
\ds\|G_\gamma\|_{H^{1/2}_{\gamma}(\omega)}
\le\frac{C}{\gamma}\left( \|(h_{1\gamma},\mathfrak{h}_{1\gamma})\|_{H^{1/2}_{\gamma}(\omega)}
+ \|\varphi_\gamma\|_{H^{1/2}_{\gamma}(\omega)} + \|\overline\varphi_\gamma\|_{H^{3/2}_{\gamma}(\omega)}
\right)
\\
\ds \le\frac{C}{\gamma}\left( \|\mathcal{F}_{\gamma} \|_{H^1_{tan,\gamma}(Q^+) }
 + \|\overline\varphi_\gamma\|_{H^{3/2}_{\gamma}(\omega)}
\right),
\end{multline}
where we have applied \eqref{estimsol2} (with $\overline\varphi$ in place of $\varphi$) and \eqref{129}. Thus, from \eqref{equaphi} again, we obtain the estimate
\begin{multline}
\label{estimgradphi}
\ds \|\nabla_{t,x'}\varphi_\gamma\|_{H^{1/2}_{\gamma}(\omega)}
\le C\left( \|(u_{1\gamma},h_{1\gamma},\mathfrak{h}_{1\gamma})\|_{H^{1/2}_{\gamma}(\omega)}
+ \|\varphi_\gamma\|_{H^{1/2}_{\gamma}(\omega)} + \|G_\gamma\|_{H^{1/2}_{\gamma}(\omega)}
\right)
\\
\ds \le\frac{C}{\gamma}\left( \|\mathcal{F}_{\gamma} \|_{H^1_{tan,\gamma}(Q^+) }
 + \|\overline\varphi_\gamma\|_{H^{3/2}_{\gamma}(\omega)}
\right).
\end{multline}
Combining \eqref{estimsol2} (with $\overline\varphi$ in place of $\varphi$), \eqref{129} and \eqref{estimgradphi} gives
\begin{equation}
\label{estimgradphi2}
\ds \|\varphi_\gamma\|_{H^{3/2}_{\gamma}(\omega)}
 \le\frac{C}{\gamma}\left( \|\mathcal{F}_{\gamma} \|_{H^1_{tan,\gamma}(Q^+) }
 + \|\overline\varphi_\gamma\|_{H^{3/2}_{\gamma}(\omega)}
\right).
\end{equation}
%%%%%%%%%%%%%%%%%%%%%%%%%%%%%%%%%%%%%%%%%%%%%%%
%%%%%%%%%%%%%%%%%%%%%%%%%%%%%%%%%%%%%%%%%%%%%%%
This defines a map $\overline\varphi\to\varphi$ in $e^{\gamma t} H^{3/2}_{\gamma}(\omega_T)$. Let $\overline\varphi^1, \overline\varphi^2\in e^{\gamma t} H^{3/2}_{\gamma}(\omega_T)$, and $(\mathcal{U}^1,W^1), (\mathcal{U}^2,W^2)$, $\varphi^1, \varphi^2 $ be the corresponding solutions of \eqref{equa1}, \eqref{equa2}, respectively. Thanks to the linearity of the problems \eqref{equa1}, \eqref{equa2} we obtain, as for \eqref{estimgradphi2},
\begin{equation*}
\ds\|\varphi^1_\gamma-\varphi^2_\gamma\|_{H^{3/2}_{\gamma}(\omega)}
\ds\le\frac{C}{\gamma} \|\overline\varphi^1_\gamma- \overline\varphi^2_\gamma\|_{H^{3/2}_{\gamma}(\omega)}.
\end{equation*}
Then there exists $\gamma_0>0$ such that for all $\gamma\ge\gamma_0$ the map $\overline\varphi\to\varphi$ has a unique fixed point, by the contraction mapping principle. The fixed point $\overline\varphi=\varphi$, together with the corresponding solution of \eqref{equa1}, provides the solution of \eqref{compactform3}, \eqref{equa2}, that is a solution of \eqref{77'}. As for the boundary conditions, we have already observed that \eqref{87},  \eqref{79'}-\eqref{82'} is equivalent to \eqref{87}, \eqref{79'}, \eqref{108}. 
The proof is complete.
\end{proof}
%%%%%%%%%%%%%%%%%%%%%%%%%%%%%%%%%%%%%%%%%%%%%%%
%%%%%%%%%%%%%%%%%%%%%%%%%%%%%%%%%%%%%%%%%%%%%%%
%%%%%%%%%%%%%%%%%%%%%%%%%%%%%%%%%%%%%%%%%%%%%%%

\section{Proof of Theorem \ref{main}}\label{proofmain}

For all $\varepsilon$ sufficiently small, problem \eqref{77} admits a unique solution with the regularity described in Theorem \ref{existenceH1}. Due to the uniform a priori estimate \eqref{54'} we can estract a subsequence weakly convergent to functions $(\mathcal{U},W,\varphi)$ with
$(\mathcal{U}_\gamma,W_\gamma)\in{H^{1}_{tan,\gamma}(Q^+_T)}\times{H^{1}_\gamma(Q^-_T)}$ and
$(q_\gamma,u_{1\gamma},h_{1\gamma})|_{\omega_T}\in{H^{1/2}_\gamma(\omega_T)}$, $W_{\gamma}|_{\omega_T}\in{H^{1}_\gamma(\omega_T)}$ and
$\varphi_\gamma\in{H^1_\gamma(\omega_T)}$ (we use obvious notations). Let us decompose $W=(\mathfrak{H},\mathfrak{E})$ and perform a inverse change of unknown with respect to that of Section \ref{equiva2} to define $(\mathcal{H},E)$ from $(\mathfrak{H},\mathfrak{E})$. 
Passing to the limit in \eqref{35"}, \eqref{77}--\eqref{87} as $\varepsilon\to0$ immediately gives that $({U},\mathcal{H},\varphi)$ is a solution to \eqref{34'new}, \eqref{93}, \eqref{95} and $E=\mathfrak{E}=0$. Passing to the limit in \eqref{54'} gives the a priori estimate \eqref{54}. The proof of Theorem \ref{main} is complete.
%%%%%%%%%%%%%%%%%%%%%%%%%%%%%%%%%%%%%%%%%%%%%%%%%
%%%%%%%%%%%%%%%%%%%%%%%%%%%%%%%%%%%%%%%%%%%%%%%%%
%%%%%%%%%%%%%%%%%%%%%%%%%%%%%%%%%%%%%%%%%%%%%%%%%%%%%
%%%%%%%%%%%%%%%%%%%%%%%%%%%%%%%%%%%%%%%%%%%%%%%%%%%%%
%%%%%%%%%%%%%%%%%%%%%%%%%%%%%%%%%%%%%%%%%%%%%%%%%%%
%%%%%%%%%%%%%%%%%%%%%%%%%%%%%%%%%%%%%%%%%%%%%%%%%%%
\section{Equivalence of systems \eqref{35"} and \eqref{secsym2}
}\label{equival}

We prove the equivalence of systems \eqref{35"} and \eqref{secsym2} for every $\vec\nu\not=0$. This is the same as the equivalence of \eqref{78} and \eqref{50"}, or \eqref{78'} and \eqref{50"}.
\begin{lemma}\label{equivalence}
Assume that systems \eqref{35"} and \eqref{secsym2} have common initial data satisfying the constraints
$$
{\rm div}\,{\mathfrak{h}}^{\varepsilon}=0,\quad {\rm div}\,{\mathfrak{e}}^{\varepsilon}=0\qquad\mbox{in }\Omega^-\quad\mbox{for}\ t=0.
$$
Assuming that the corresponding Cauchy problems for \eqref{35"} and \eqref{secsym2} have a unique classical solution on a time interval $[0,T]$, then these solutions coincide on $[0,T]$ for all ${\varepsilon}$ sufficiently small.

\end{lemma}
\begin{proof}
Let us set
$$
A=\hat\eta^{-1}(\partial_t  \mathfrak{h}^{\varepsilon}+{\varepsilon}^{-1}\nabla\times \mathfrak{E}^{\varepsilon}), \qquad
B=\hat\eta^{-1}(\partial_t  \mathfrak{e}^{\varepsilon}-{\varepsilon}^{-1}\nabla\times \mathfrak{H}^{\varepsilon}).
$$
Then \eqref{secsym2} can be written as
\begin{equation}
\begin{array}{ll}\label{secsym3}
\displaystyle
A -
\varepsilon \, \vec \nu\times B +
\frac{\vec\nu}{\partial_1\widehat{\Phi}_1}\,{\rm div}\,\mathfrak{h}^{\varepsilon}=0,
\qquad
B +
\varepsilon \, \vec\nu\times A +
\frac{\vec\nu}{\partial_1\widehat{\Phi}_1}\, {\rm div}\,\mathfrak{e}^{\varepsilon}=0.
\end{array}
\end{equation}
Taking the vector product of ${\vec\nu}$ with the systems in \eqref{secsym3} gives
\begin{equation}
\begin{array}{ll}\label{secsym41}
\displaystyle
\vec \nu\times A -
\varepsilon \, \vec \nu\times (\vec \nu\times B)=0,
\qquad
\vec \nu\times B +
\varepsilon \, \vec \nu\times (\vec \nu\times A)=0,
\end{array}
\end{equation}
that is
\begin{equation}
\begin{array}{ll}\label{secsym4}
\displaystyle
\vec \nu\times A -
\varepsilon \, (\vec\nu\cdot B)\vec \nu + \varepsilon \, |\vec\nu|^2 B=0,
\qquad
\vec \nu\times B +
\varepsilon \, (\vec\nu\cdot A)\vec \nu - \varepsilon \, |\vec\nu|^2 A=0.
\end{array}
\end{equation}
We take the vector product of $\varepsilon \, {\vec\nu}$ with the first system in \eqref{secsym4} and get
\begin{equation}
\begin{array}{ll}\label{secsym5}
\displaystyle

\varepsilon \, (\vec\nu\cdot A)\vec \nu - \varepsilon \, |\vec\nu|^2 A + \varepsilon^2 \, |\vec\nu|^2 \, \vec \nu\times B=0.
\end{array}
\end{equation}
For any choice of $\vec \nu\not= 0$ we may assume that $\varepsilon \, |\vec\nu|\not=1$ (this is true for $\varepsilon$ definitely small). Then by comparison of \eqref{secsym5} and the second equation in \eqref{secsym4} we infer $\vec \nu\times B =0$, and from \eqref{secsym41} also $\vec \nu\times A =0$.

Thus \eqref{secsym2} may be rewritten as
\begin{equation*}
\begin{array}{ll}\label{secsym22}
\displaystyle
\partial_t  \mathfrak{h}^{\varepsilon}+\frac{1}{\varepsilon}\nabla\times \mathfrak{E}^{\varepsilon} +\frac{\hat\eta\,\vec\nu}{\partial_1\widehat{\Phi}_1}\,{\rm div}\,\mathfrak{h}^{\varepsilon}=0,
\qquad
\displaystyle
\partial_t  \mathfrak{e}^{\varepsilon}-\frac{1}{\varepsilon}\nabla\times \mathfrak{H}^{\varepsilon} +
\frac{\hat\eta\,\vec\nu}{\partial_1\widehat{\Phi}_1}\, {\rm div}\,\mathfrak{e}^{\varepsilon}=0.
\end{array}
\end{equation*}
Applying the {div} operator to the equations gives the transport equation
$$
\partial_t  u+{\rm div}(u\vec a)=0 \quad \mbox{ in } Q^-_T,
$$
for both $u={\rm div}\,{\mathfrak{h}}^{\varepsilon}$ and $u={\rm div}\,{\mathfrak{e}}^{\varepsilon}$, where $\vec a=\hat\eta \vec\nu/{\partial_1\widehat{\Phi}_1}$.
Noticing that the first component of $\vec a$ vanishes at $x_1=0$, the transport equation doesn't need any boundary condition.
As $u_{|t=0}=0$, by a standard argument we deduce $u=0$ for $t>0$.
This fact shows the equivalence of \eqref{35"} and \eqref{secsym2}.

\end{proof}

%\vfil
%\eject

%%%%%%%%%%%%%%%%%%%%%%%%%%%%%%%%%%%%
%%%%%%%%%%%%%%%%%%%%%%%%%%%%%%%%%%%%
%%%%%%%%%%%%%%%%%%%%%%%%%%%%%%%%%%%%
%%%%%%%%%%%%%%%%%%%%%%%%%%%%%%%%%%%%%
\section{Proof of Lemma \ref{lemma1}}
\label{prooflemma1}

Given an even function $\chi\in C^\infty_0(\R)$, with $\chi=1$ on $[-1,1]$, we define
\begin{equation}
\label{def-f1}
\Psi(x_1,x') :=\chi (x_1\langle D\rangle) \, \varphi(x') \, ,
\end{equation}
where $\chi(x_1 \langle D\rangle)$ is the pseudo-differential operator with $\langle D\rangle=(1+|D|^2)^{1/2}$ being the Fourier multiplier in the
variables $x'$. From the definition it readily follows that $\Psi(0,x')=\varphi(x')$ for
all $x'\in\R^2$. Moreover,
\begin{equation}
\label{d1Psi}
\duno \Psi(x_1,x') = \chi'(x_1\langle D\rangle) \, \langle D\rangle \, \varphi(x') \, ,
\end{equation}
which vanishes if $x_1=0$. We compute
\begin{equation*}
\|\Psi(x_1,\cdot)\|^2_{H^m(\R^2)} =\int_{\R^2}\langle \xi'\rangle^{2m}\chi^2(x_1\langle \xi'\rangle)|\hat\varphi(\xi')|^2d\xi' \, ,
\end{equation*}
where $\hat\varphi(\xi')$ denotes the Fourier transform in $x'$ of $\varphi$.
It follows that
\begin{align*}
\| \Psi \|^2_{L^2_{x_1}(H^m(\R^2))} =\int_{\R}\int_{\R^2} \langle \xi'\rangle^{2m}\chi^2(x_1\langle \xi'\rangle)|\hat\varphi(\xi')|^2d\xi' \,dx_1
\\
=\int_{\R}\int_{\R^2} \langle \xi'\rangle^{2m-1}\chi^2(s)|\hat\varphi(\xi')|^2d\xi' \,ds
\le C\|\varphi\|^2_{H^{m-0.5}(\R^2)}\, .
\end{align*}
In a similar way, from \eqref{d1Psi}, we obtain
\begin{align*}
\| \duno \Psi \|^2_{L^2_{x_1}(H^{m-1}(\R^2))} =\int_{\R}\int_{\R^2} \langle \xi'\rangle^{2m-2}|\chi'(x_1\langle \xi'\rangle)\langle \xi'\rangle|^2|\hat\varphi(\xi')|^2d\xi' \,dx_1
\\
=\int_{\R}\int_{\R^2} \langle \xi'\rangle^{2m-1}|\chi'(s)|^2|\hat\varphi(\xi')|^2d\xi' \,ds
\le C\|\varphi\|^2_{H^{m-0.5}(\R^2)}\, .
\end{align*}
Iterating the same argument yields
\begin{equation*}
\| \duno^j \Psi \|^2_{L^2_{x_1}(H^{m-j}(\R^2))} \le C\, \| \varphi \|_{H^{m-0.5}(\R^2)}^2 \, ,\quad j=0,\dots,m \, .
\end{equation*}
Adding over $j=0,\dots, m$ finally gives $\Psi \in H^m(\R^3)$ and the continuity of the map $\varphi \mapsto \Psi$.

We now show that the cut-off function $\chi$, and accordingly the map $\varphi \mapsto \Psi$, can be chosen to give \eqref{diffeopiccolo}. From \eqref{d1Psi} we have
\begin{equation*}
\label{}
\duno \Psi(x_1,x') = (2\pi)^{-2}\int _{\R^2} e^{i\xi'\cdot x'}\chi'(x_1\langle \xi'\rangle) \, \langle \xi'\rangle \, \hat\varphi(\xi') \, d\xi' .
\end{equation*}
By the Cauchy-Schwarz inequality and a change of variables we get
\begin{align*}
\label{}
|\duno \Psi(x)| \le C \| \varphi\|_{H^{2}(\R^2)} \left( \int _{\R^2} |\chi'(x_1\langle \xi'\rangle)|^2 \, \langle \xi'\rangle^{-2} \, \, d\xi' \right)^{1/2}
\\
\qquad\qquad =C \| \varphi\|_{H^{2}(\R^2)} \left( \int _0^{\infty} |\chi'(x_1\langle \rho\rangle)|^2 \, \langle \rho\rangle^{-2} \, \rho\, d\rho \right)^{1/2} .
\end{align*}
We change variables again in the integral above by setting $s=x_1 \langle \rho\rangle$. It follows that
\begin{align}
\label{d1Psi2}
|\duno \Psi(x)| \le C \| \varphi\|_{H^{2}(\R^2)} \left( \int _{x_1}^{\infty} |\chi'(s)|^2 \, \frac{x_1}{s}\, \frac{ds}{x_1} \right)^{1/2}
\le C \| \varphi\|_{H^{2}(\R^2)} \left( \int _{1}^{\infty} |\chi'(s)|^2 \, \frac{ds}{s} \right)^{1/2} .
\end{align}
Given any $M>1$, we choose $\chi$ such that $\chi(s)=0$ for $|s|\ge M$, and $|\chi'(s)|\le 2/M$ for every $s$. Then from \eqref{d1Psi2} one gets
\begin{align*}
\label{d1Psi3}
|\duno \Psi(x)| \le \frac{C}{\sqrt{M}} \| \varphi\|_{H^{2}(\R^2)} .
\end{align*}
Given any $\eps>0$, if $M$ is such that $C/\sqrt{M}<\eps$, then \eqref{diffeopiccolo} immediately follows.

\medskip

The proof of Lemma \ref{lemma2} follows from Lemma \ref{lemma1}, with $t$ as a parameter. Notice also that
the map $\varphi\to \Psi$, defined by  \eqref{def-f1}, is linear and that the time regularity is conserved because, with
obvious notation, $\Psi(\dt^j \varphi) =\dt^j \Psi(\varphi)$. The conclusions of Lemma \ref{lemma2} follow directly.

\section{Proof of Lemma \ref{interpol}}\label{proofinterpol}

We write $\mathcal{U}_{n\gamma}^{\varepsilon}$ on $\omega$ as
$$
|\mathcal{U}_{n\gamma}^{\varepsilon}|^2|_{x_1=0} = - 2\int_0^\infty \mathcal{U}_{n\gamma}^{\varepsilon}\cdot\partial_1\mathcal{U}_{n\gamma}^{\varepsilon}\, dx_1,
$$
which gives
\begin{equation}
\begin{array}{ll}\label{interpol3}
\|\mathcal{U}_{n\gamma}^{\varepsilon}|_\omega\|^2_{L^2(\omega)} \le 2\| \mathcal{U}_{\gamma}^{\varepsilon}\|_{L^2(Q^+)}
\|\partial_1\mathcal{U}_{n\gamma}^{\varepsilon}\|_{L^2(Q^+)}.
\end{array}
\end{equation}
Estimating the right-hand side of \eqref{interpol3} with \eqref{130} and using the $\gamma$-homogeneity of the $H^1_{tan,\gamma}$ norm gives
$$
\gamma\|\mathcal{U}_{n\gamma}^{\varepsilon}|_\omega\|^2_{L^2(\omega)}
\le C\left(\|\mathcal{F}_\gamma\|^2_{L^2(Q^+)}
 +\|\mathcal{U}_\gamma^{\varepsilon}\|^2_{H^1_{tan,\gamma}(Q^+)}
\right).
$$
Thus the first part of \eqref{interpol1} is proved. To show the second part of \eqref{interpol1} we compute for ${\ell}=0,2,3$,
$$
\int_\omega|Z_{\ell}\mathcal{U}_{n\gamma}^{\varepsilon}|^2|_{x_1=0}\, dx'dt
= - 2\int_0^\infty\int_\omega Z_{\ell}\mathcal{U}_{n\gamma}^{\varepsilon}\cdot\partial_1Z_{\ell}\mathcal{U}_{n\gamma}^{\varepsilon}\, dxdt
=  2\int_0^\infty\int_\omega Z_{\ell}^2\mathcal{U}_{n\gamma}^{\varepsilon}\cdot\partial_1\mathcal{U}_{n\gamma}^{\varepsilon}\, dxdt,
$$
which gives
\begin{equation}
\begin{array}{ll}\label{interpol4}
\|\mathcal{U}_{n\gamma}^{\varepsilon}|_\omega\|^2_{H^1_{\gamma}(\omega)} \le 2\| \mathcal{U}_{\gamma}^{\varepsilon}\|_{H^2_{tan,\gamma}(Q^+)}
\|\partial_1\mathcal{U}_{n\gamma}^{\varepsilon}\|_{L^2(Q^+)}.
\end{array}
\end{equation}
Interpolating between \eqref{interpol3} and \eqref{interpol4} gives
\begin{equation*}
\begin{array}{ll}\label{interpol5}
\|\mathcal{U}_{n\gamma}^{\varepsilon}|_\omega\|^2_{H^{1/2}_{\gamma}(\omega)} \le 2\| \mathcal{U}_{\gamma}^{\varepsilon}\|_{H^1_{tan,\gamma}(Q^+)}
\|\partial_1\mathcal{U}_{n\gamma}^{\varepsilon}\|_{L^2(Q^+)}.
\end{array}
\end{equation*}
Applying \eqref{130} eventually gives the second part of \eqref{interpol1}. We do the same for \eqref{interpol2}. 

\bibliographystyle{plain}
\bibliography{SeTr}
\end{document}